\documentclass[a4paper,11pt]{article}

\usepackage{amsmath,amssymb,enumerate,xspace,amsbsy,amsfonts,graphics,epsfig,float,mathrsfs,amsthm,xcolor,latexsym}
\usepackage[T1]{fontenc}
\usepackage[utf8]{inputenc}
\usepackage{graphicx}

\usepackage[bookmarks]{hyperref}

\newcommand{\tr}{\mathop{\mathrm{tr}}\nolimits}
\newcommand{\GL}{\mathop{\mathrm{GL}}\nolimits}

\newcommand{\FQ}{\mathfrak {F}(Q)}

\renewcommand{\v}[1]{\mathbf{#1}}

\renewcommand{\c}[1]{\textsf{#1}}

\newcommand{\pd}[2]{\frac{\partial #1}{\partial #2}}

\newcommand{\pddd}[3]{\frac{\partial^2 #1}{\partial {#2} \partial {#3}}}
\newcommand{\sd}[2]{\frac{\textrm{d} #1}{\textrm{d} #2}}

\newcommand{\be}[1]{\begin{equation}\label{#1}}
\newcommand{\ee}{\end{equation}}
\newcommand{\bes}[1]{\begin{equation}\label{#1} \begin{split}}
\newcommand{\ees}{\end{split} \end{equation}}
\newcommand{\ben}[1]{\begin{eqnarray}\label{#1}}

\newcommand{\re}[1]{~(\ref{#1})}

\newcommand{\bma}{\begin{pmatrix}}
\newcommand{\ema}{\end{pmatrix}}

\renewcommand{\bar}[1]{\overline{#1}}

\newcommand{\h}{\mathcal {H}}

\newtheorem{theo}{Theorem}
\newtheorem{prop}[theo]{Proposition}
\newtheorem{proposition}[theo]{Proposition}

\newtheorem{lem}[theo]{Lemma}
\theoremstyle{definition} \newtheorem{defi}{Definition}
\newtheorem{rem}{Remark}

\begin{document}
\thispagestyle{empty}
\title{The two momenta of an elastic rod:\\ a Hamiltonian picture on framed Lie groups}
\author{T. Lessinnes \\ \'Ecole Polytechnique de Bruxelles, Universit\'e Libre de Bruxelles, TIPs}

\maketitle

\hrule
\vskip 6pt
\noindent {\bf\small $\blacksquare$ Abstract\ } {\small
The equilibrium equations of elastic rods can be obtained by balancing forces
and moments, or by rendering a potential energy stationary. For complex filaments the energy route asks less of one's mechanical intuition. However, in the classical Hamiltonian picture, the components of the generalized momenta are postulated \emph{a priori} and their physical meaning
changes at the
whim of the coordinate chart. Here, we draw on two ideas from mathematical physics. First, the extended tangent bundle of the configuration
Lie group is framed by left-invariant vector fields. Second, we follow a one-dimensional reading of the Cartan--Lepage--Krupka
theory of variational forms: in this setting
the momenta are not postulated but forced --- they are the unique functions
for which the governing two-form is nonsingular with holonomic
characteristics. The Legendre transform becomes a linear change of
frame, and a Poisson structure on the extended tangent bundle follows.

For an isolated Cosserat rod, the momenta coincide, in every encoding of the
rotation group, with the material force and moment $(\c n,\c m)$ familiar from
rod theories. In the presence of interactions, two cases arise: interactions
that depend only on the configuration, such as gravity, leave this
identification intact; interactions that depend on the strains destroy it --- the conjugate momenta and the internal stresses part company. Because energies
add, the momenta decompose into the internal stresses and an interaction
contribution. Expressing the two factors of the Poisson bivector in different
frames --- one adapted to the internal stresses, the other to the conjugate
momenta --- then exposes the Hamiltonian flow directly in the internal variables $(q,\c n,\c m)$, where $q$ encodes the configuration.

Applied to a tendon-actuated rod, where the standard passage to the
Hamiltonian picture demands a nonlinear inversion, the construction delivers explicit 
equilibrium equations, the required inversion collapsing to a rank-one correction. The collapse is structural: a tendon
contributes its length to the energy, and the Hessian of a length has rank one.

}\vskip 6pt
\hrule

\tableofcontents

\section{Introduction}

The Hamiltonian picture has remained a minority practice in rod statics ---
pursued in depth by the Maddocks school~\cite{dilima95,kema97b} and by van der
Heijden and coworkers~\cite{va01,sthe14}, but rarely adopted elsewhere. The
reason may be practical. As soon as the rod interacts with its surroundings
through its strains, the passage to the Hamiltonian picture demands the
inversion of a nonlinear relation between strains and momenta, and the
resulting equations compare poorly with what a careful balance of forces and
moments delivers. This paper removes that obstruction. We show that, cast in
the right geometric setting, the conjugate momenta are not chosen but forced,
and that for a rod on its own they are precisely the material force and moment
of classical rod theory --- in every encoding of the rotations. When an
interaction couples to the strains this identification breaks, in a way the
framework makes exact and computable. Applied to a
tendon-actuated rod, this approach provides the equilibrium equations in closed form. The cost is
a change of viewpoint: frames instead of coordinates, and a variational
construction in which the momenta are derived rather than declared.

The theory describing the statics of thin elastic filaments, also known as rods, has a long history. Euler's study~\cite{eu44} of the equilibria of the elastica (a two-dimensional rod) followed by Kirchhoff's analogy~\cite{ki59} for the three-dimensional case and the Cosserat brothers' analysis~\cite{Cosserat1909} allowing for extensibility and shearability are early milestones. A modern exposition can be found in~\cite{an05}. One standard approach balances forces and moments applied to any segment of the rod. 
Another classical approach encodes the potential energy of the system as the sum of an elastic energy density $W$ (per unit reference length) stored in the rod and  the potential energy associated with its interaction with the world. This total energy is expressed as an integral to be performed along the length of the structure. Balance equations are then recovered \emph{via} the Euler-Lagrange equations associated with this functional.

The first approach is direct. However, for complex filaments it is often easier to postulate the  potential energy of the system. A few examples of complex filaments are 2-braids~\cite{sthe14}, $n$-plies~\cite{neva02}, birods~\cite{moma05,lemogo14} or $n$-rods~\cite{moulton2020morphoelastic}, but also constrained filaments such as a rod constrained to stay on a cylinder~\cite{va01,vachth02}. In all these cases, dissecting the balance of forces and moments between components is challenging. The problem is that going from the energy functional to the balance equation is a non-trivial procedure. And depending on the choice of coordinates for the configuration space $Q$, the ensuing Euler-Lagrange equations may not look like the well-studied Cosserat and Kirchhoff equations. 
One neat way of doing this for Cosserat rods is described in~\cite{ch03}. The trick is to encode the perturbation of curves in a way that respects the group structure of the rotating material frames.

Aiming for the generality that covers both simple Cosserat rods and complex filaments, a rod is defined here as a continuous one-parameter family of sections. Each of these sections can assume different configurations. Let $Q$ be the set of all such configurations; for all the rods mentioned above, $Q$ is a manifold. But there is more: $Q$ is a Lie group because configurations can be mapped to their associated  transformations from some reference configuration.
And more yet, Kirchhoff first stated that neglecting the intricate deformation within the section is the move that transforms the three-dimensional body of the rod that has small but finite thickness into a one-dimensional problem with each section having a finite number of degrees of freedom; six of them in his case. For complex rods, that number of degrees of freedom increases, but the main idea behind rod theories is that it stays finite. In short, the configuration set $Q$ is a \emph{finite-dimensional} Lie group.

Since we are dealing with a variational problem on curves, a curious physicist would undoubtedly wonder what could be gained by studying the Hamiltonian picture. After all, Kirchhoff himself started by noting that the balance equations for the statics of the unshearable and inextensible rod look a lot like the differential equations governing the dynamics of the heavy top; a system that has been extensively studied by Hamiltonian techniques. She should turn to~\cite{dilima95} where it was noted that for Cosserat rods, the Poisson bracket structure depends on the choices made to encode the $SO(3)$ factor in $Q=SE(3)$. One can read their push towards a description by quaternions as an attempt to get a better bivector within the Poisson structure. The fact that the (components of the) generalized momenta depend on the choice of coordinates is in general an inconvenience that all but ensures that the resulting balance equations will look quite foreign to the original Kirchhoff/Cosserat picture. Here we propose to avoid this by casting the problem in a way that respects left-invariance of vector fields on the configuration group $Q$. In a sense the idea is a generalization to arbitrary groups of the trick used in~\cite{ch03}. But it  is in fact much older than that.

Harnessing left-invariance for variational problems on Lie groups originates with Poincar\'e~\cite{poincare01}. The corresponding Hamiltonian structure  on Lie groups traces back to Lie himself. The subsequent
coadjoint-orbit reading was developed by Arnold, Kirillov, Kostant and
Souriau~\cite{ki62,ar66, ko70b,so70,ki04}. Its systematic development through
\emph{reduction} came later~\cite{maje99,cema01}. In the reduction approach, a
configuration space $Q$ carrying a symmetry group $G$ is quotiented, and the
dynamics descend to $(T^\star Q)/G$ --- the Lie--Poisson structure on
$\mathfrak g^\star$ when $Q=G$, and more generally the Hamilton--Poincar\'e
equations on a bundle whose geometry is encoded by a principal connection and its
curvature~\cite{cema03}. Loadings that depend on the configuration, and so break
the invariance, can be brought back into this picture by enlarging the group with
advected parameters~\cite{homa98}.

This reduction technique was brought to bear on rods in~\cite{kema97b} with a new spin on the Kirchhoff analogy: the Marsden and Ratiu school had treated the heavy top in some detail. So the approach was directly translated to rods. But there as well, when performing actual rod integration, specific choices of coordinates led to specific Poisson brackets and Hamiltonian pictures with different numbers of Casimirs, integrals and involutions (see Table 1 in~\cite{kema97b}).

We pursue a different tack with a different technical budget. We remain in the extended tangent bundle $\overline{TQ}:= TQ\times[0,L]$: the Cartesian product between the tangent bundle to $Q$ and the interval $[0,L]\subset\mathbb R$ of (reference) arc length values containing the rod. The Lie group structure is preserved by expanding vector fields and forms
in a left-invariant frame and its dual, but without quotienting, without choosing a
connection, and without introducing advected quantities to restore configuration
dependence: the configuration is simply never removed.

This casting of the variational problem disconnects the conjugate momenta and the Poisson structure from any choice of coordinates on $Q$. The bivector that encodes the Poisson structure on $\overline{TQ}$ can be pushed forward to any coordinate chart or group representation of $Q$ where we could then recover the various Poisson brackets studied in~\cite{kema97b} for Kirchhoff rods. Actually doing so is not our aim. Rather, we will display in section~\ref{sec-singlerod} how the theoretical framework directly yields Kirchhoff-like equations for rods regardless of the choices made to encode the $SO(3)$ factor. Working with frames brings us towards the Hamel equations~\cite{blma09} which encode the Euler-Lagrange equations in frames. 

Another advantage is that with frames, there is no need to invert the potentially non-linear relation between quasi-velocities and conjugate momenta. Their derivatives can be computed in whichever frame is expedient, since changing frame is a linear operation. Accordingly, the $p_i\, \xi^i$ term of the Hamiltonian need not be an explicit function of $(q,p)$.

A second choice is made here; it concerns how the Hamiltonian picture is introduced. In the standard argument the momenta $p_i:=\partial\mathcal L/\partial\dot q^i$ are \emph{introduced} by hand,
invertibility is verified, and the physical meaning of $p_i$ is then left to be discovered chart by chart. Furthermore, after the Legendre transform, what is $\int_\Gamma p_i dq^i -\mathcal H ds$? Is it still an energy you can \emph{minimize}? Instead, we follow another route: Lepage equivalent forms. These ideas were developed in a chain of works by Cartan~\cite{ca22}, Lepage~\cite{le36}, Krupka~\cite{kr73} and others (see~\cite{kr15} for a modern exposition) in the context of mathematical physics. Their targets were systems the state of which is encoded as fields defined over a space-time of arbitrary dimension $k$. For the statics of rods $k=1$ (because filaments are inherently one-dimensional) and it is possible to pose the problem along similar ideas without the need to refer to the full mathematical complexity required for the general treatment.
The distinction is not merely {\ae}sthetic; it has two concrete merits. First, the
momenta are \emph{derived} rather than declared. They are the unique functions rendering the governing form nonsingular with
holonomic characteristics~(see section~\ref{sec-lepagecartan}). In a left-invariant frame adapted to the 
material directors the functions it returns are exactly the material force and
moment $(\c n,\c m)$ for rods. Second, the functional retains its meaning throughout. The construction proves that the
form we extremalize \emph{is} the energy of the rod on holonomic curves.
This will matter in later work, where stable equilibria are sought not merely as
critical curves but as \emph{minimizers} of that energy.

Interactions that depend only on the configuration, such as gravity, leave
the identification of the momenta $p$ with the internal stresses $(\c n,\c m)$
intact; interactions that depend on the strains destroy it, and the momenta
$p$ and the internal stresses part company. The Lepage-equivalent approach
makes it clear that the $p_i$ are the canonical quantities. The tension then
resolves because energy is additive: the energy density splits as a sum where the first term is the internal energy of the rod
and the second encodes the interaction, and the momenta decompose
accordingly. We show how to keep the Hamiltonian structure developed above 
while expressing all quantities as functions of $\c n$ and $\c m$. Here again, a judicious choice of frame is key. The Poisson bracket and the
Hamiltonian function on $\overline{TQ}$ are left unchanged, but the bivector
of the bracket is expressed with its two factors in different frames: the
first in the frame associated with $(q,\c n,\c m)$, the second in the frame
associated with $(q,p)$. We end by showcasing the technique on the analysis
of a tendon-actuated soft robot.

The plan is as follows. Section~\ref{sec-cosserat} recalls the Cosserat rod,
fixes notation, and casts the configuration of a section as an element of a Lie
group $Q$ --- the one structural assumption the construction rests on.
Section~\ref{sec-theory} is the engine room: it builds the framed extended
tangent bundle $\overline{TQ}$, shows that the conjugate momenta are
\emph{forced} rather than chosen (Proposition~\ref{prop-pforced}), identifies the
resulting Poincar\'e--Cartan form as the energy of the rod, turns the Legendre
transform into a linear change of frame, and reads off a Poisson structure. It
closes with a step-by-step recipe (section~\ref{sec-steps}) that takes one
mechanically, group in hand, from an energy functional to the equilibrium equations.
The remaining sections put the recipe to work. Section~\ref{sec-singlerod} treats the case where the momenta $p$ align with the rod stresses.
The recipe recovers the classical Kirchhoff equations when it is applied to the familiar problem of the tip-loaded elastic rod. The same computation carried out in three encodings of
the rotation group (rotation matrices, unit quaternions, and Cayley vectors) illustrates that the answer is independent of the encoding. Section~\ref{sec-interact}
turns to strain-dependent interactions: we show how such a loading is handled by adapting the framing of each factor in the Poisson bivector. The internal
moment $\c m$ is carried through the bracket even though the conjugate momentum is
no longer $\c m$ itself. We illustrate the technique on a tendon-actuated rod,
where the required matrix inversion collapses to
a single rank-one correction and the equilibrium equations are deduced in explicit form.

\section{A starting point: Cosserat rods\label{sec-cosserat}} 

A Cosserat rod is a continuous, one-parameter collection of rigid bodies -- the sections of the rod. We therefore label each section by a material parameter $S\in[A,B]\subset\mathbb R$ that increases monotonically from one tip of the rod to the other. The position of each section is then encoded as a reference point $\v r(S)$ within the section and a material basis $(\v d_1(S),\,\v d_2(S),\, \v d_3(S))$ which is always chosen orthonormal and is locked with material directions of the section. The frame $\{\v r,\v d_1,\v d_2,\, \v d_3\}$ is called the material frame of the section $S$. Given a lab frame $\{ O, \v e_1,\,\v e_2,\, \v e_3\}$, we can therefore encode the state of a section by a translation which we also denote\footnote{There is a slight overload of notation between the position $\v r$ of the centreline and the translation $\v r$ from the origin of the lab frame to it. It is however benign and standard in the field.}  $\v r(S)$, and a rotation with rotation matrix $R(S)\in SO(3)$ from the lab basis to the material basis. 

The frame $\{\v r,\v d_i\}$ both translates and rotates as one moves along the rod. The translation rate $\v r'(S)$ and the rotation rate $\v u$ can both be expressed in the local basis as 
\begin{align} \label{basicgeom}
\v r'&=\c v^i\, \v d_i,&
\v u&=\c u^i\, \v d_i,&
R'&=R\, \c u^\times,
\end{align}
where once and for all, and in keeping with habits of the field~\cite{dilima95,goriely17}, both vectors and the column of their components in the lab frame are denoted as bold letters such as $\v u$ while the column  of their components in the local frame $\v d_i$ is denoted in sans-serif font such as $\c u$. The last equality in~\eqref{basicgeom} defines the column $\c u^i$ via $\c u^\times = R^T\, R'$ which is a skew matrix as can be seen by taking a derivative of $R^TR=$Id. We have also used the notation
\begin{align}
\label{defcross}
\c u^\times =\bma 0 & -\c u^3 &\c u^2 \\
\c u^3 & 0 & -\c u^1 \\ 
-\c u^2 & \c u^1 & 0
\ema,
\qquad\textrm{that is} \qquad 
\left(\c u^\times \right)^k_j &= \varepsilon^k_{ij} \c u^i,
\end{align}
with $\varepsilon^k_{ij}$ being fully antisymmetric with $\varepsilon^1_{23} =1$.  In particular, the last equation in~\eqref{basicgeom} is equivalent to $\v d_i'= \v u\times \v d_i$.

In the absence of loads of any kind, the rod assumes a reference shape encoded by $\c v = \widehat{\c v} := \bma 0 & 0 & 1\ema^T$ and $\c u = \widehat {\c u}$. The specific choice of $\widehat {\c v}$ expresses that without load, the centreline $\v r$ runs along the rod with $\v d_3$ tangent to it and that $S$ is the arc length parametrization in that reference shape. The reference curvature $\widehat {\c u}$ encodes for a rod that would be coiled in its reference shape. 

\subsection*{Forces and moments balance approach}

For any section $S^\star$, picture mentally splitting the rod into two parts: the material lying at $S>S^\star$ and that which is at $S\leq S^\star$. The first part applies a force $\v n$ and a moment of force $\v m$ with respect to $\v r(S^\star)$ on the second part. To complete the picture, we assume that the rod is under a density of force $\v f$ per unit $S$ and a density of moment $\v l$ per unit $S$ and with respect to $\v r(S)$.

Hence, in a static configuration, the balance  of forces and moments with respect to $\v r(S^\star)$  on the piece of the rod referenced by the interval $S\in[S^\star ,S^\star + \varepsilon]$ reads
\begin{align}
&\v n(S^\star+ \varepsilon) -  \v n(S^\star) + \int_{S^\star}^{S^\star + \varepsilon} \v f(S)\, dS= 0,\\
&\v m(S^\star+ \varepsilon) -  \v m(S^\star) + \left[ \v r(S^\star+ \varepsilon) - \v r(S^\star)\right] \times \v n(S^\star+\varepsilon)\\
&\qquad\qquad\qquad+  \int_{S^\star}^{S^\star + \varepsilon}\left[\v  r(S) - \v r(S^\star)\right] \times \v f(S) \, dS+ \int_{S^\star}^{S^\star + \varepsilon} \v l(S) \, dS= 0.
\end{align}
Dividing by $\varepsilon$ and then taking $\varepsilon \to 0$, yields the Kirchhoff equations~\cite{ki59} for rod statics:
\begin{align}\label{cosseratbalance}
\v n'(S^\star) + \v f(S^\star) &= 0,& 
\v m'(S^\star) +\v r'(S^\star) \times \v n(S^\star) + \v l(S^\star) &= 0.
\end{align}

To close this set of equations, we need to express constitutive laws that fix the force $\v n$ and moment $\v m$ as functions of the shape of the rod. Most often, the relation between the momenta $\c m$ and the strains $\c u-\widehat{\c u}$ is assumed linear: 
\begin{align}
\c m_i = K_{ij}\, (\c u-\widehat{\c u})^j,
\end{align}
where $K$ is a symmetric and positive-definite constitutive matrix. Throughout this text, matrix elements with their indices on the same level are listed as row index first followed by column index; matrix elements with both upper and lower indices (like in $\mathcal G^i_j$ in section~\ref{sec-interact}) have row index up. The material frame is systematically chosen so that $K$ is diagonal~\cite{an05}. 

As for the force $\c n$, many practical applications assume a Kirchhoff rod. That is, the rod is unshearable (the first two components of $\c v$ always vanish) and inextensible ($\c v^3 = 1$; the reference arc length parametrization remains an arc length parametrization in all configurations). These three constraints on $\c v$ close the system of equations so that no further constitutive laws should be enforced on $\v n$. 

In general, the balance laws~\eqref{cosseratbalance} may also be cast in the material frame. For a general vector $\v w$, we have $\v w' = \left(\c w^i\,\v d_i\right)' = (\c w' +\c u\times \c w)^i\,\v d_i$. The full set of equations for the Kirchhoff rod is therefore 
\begin{align}\label{kirchhoff}
&\c n' + \c u\times \c n + \c f = 0, &
&\c m' +\c u\times \c m + \widehat{\c v} \times \c n +\c l= 0,& 
 &\v r' = \v d_3, &
& \v d_i' = \v u\times \v d_i &
 &\c m = K(\c u-\widehat {\c u}).
\end{align}

Finally, there are applications for which we do away with the inextensibility, and more rarely, the unshearability assumptions. In those cases, a constitutive law for the force is specified, and it too is often assumed linear: $\c n = A\,(\c v -\widehat{\c v}).$ Once again the stiffness matrix $A$ is assumed to be symmetric and positive definite. In many practical cases, we may have both $A$ and $K$ diagonal at the same time. The unshearability condition amounts to taking the limits to infinity of the first two diagonal entries of $A$ and the inextensibility condition is recovered by taking the limit to infinity of the third diagonal entry of $A$.

\subsection*{The variational approach}

The Cosserat rods described above were modelled by imposing the principle of conservation of momentum and assuming particular constitutive laws. It is possible to deduce the same equations~\eqref{kirchhoff} by defining equilibria as critical curves of a potential energy functional. To do so, one posits an energy density function $W\left(\c v-\widehat{\c v},\,\c u-\widehat{\c u}\right)$ that encodes the elastic energy density per unit $S$ stored in the rod when it is assuming the shape encoded by $\c v$ and $\c u$. If the loads $\v f $ and $\v l$ are themselves conservative, there exists a potential energy density $\psi(\v r,R)$ associated with them. The potential energy functional of the system is then
\begin{equation}\label{introenergyfunc}
  E = \int_0^L W(\c v-\widehat{\c v},\c u -\widehat{\c u}) + \psi(\v r,R)\, dS.
\end{equation}
Equilibria are defined as the solutions of the associated Euler-Lagrange equations. Chapter 6 of~\cite{ch03} expounds a  crafty technique for taking a variation of~\eqref{introenergyfunc} that is well worth knowing for any rod practitioner. Using it regularly was an inspiration for a central idea of the present document: we ought to be working with frames (more on that in section~\ref{sec-frame}).

\subsection*{Generalised filaments}

For Cosserat and Kirchhoff rods described above the configuration of a section is encoded by a translation and a rotation from some reference frame. That is the configurations can be encoded by the Lie group\footnote{Where it is usual to overload $SE(3)$ to mean both the abstract group of orientation-preserving isometries and its faithful $\GL(4,\mathbb{R})$ representation.} $Q=SE(3)$ consisting of orientation preserving isometries of the Euclidean three-dimensional space. 

Complex filaments are constituted of bundles of smaller rods attached to each other in various ways. Accordingly, generalized filaments are continuous one-parameter collections of objects whose configurations can be encoded by elements of a finite-dimensional Lie group. In the case of bundles of $k$ subfilaments, we can always do this by considering $SE(3)^k$. Sometimes a smaller group will do. In the Bristol ladder~\cite{pila13} two flexible flanges are connected by rigid spokes; a working configuration group is $SO(3)\times SO(2)$. The first factor handles one of the flanges and the second may encode the rotation of the sister flange with respect to the first one.

The group of all possible transformations encodes geometric information relevant to a specific class of problems that share that group. And that is true before constitutive laws are provided. Even though it is always possible to choose coordinates on (patches of) any such group and simply work in $\mathbb R^n$, precious information is lost in the process that will typically need to be reintroduced later by clever tricks. In what follows, this information is preserved by expanding vector fields (and 1-forms)  on a frame that is built upon  left-invariant vector fields on $Q$. 

We develop the general picture in section~\ref{sec-theory}. Table~\ref{tab-translate} is there to keep the connection between the rods described here and the machinery to come. The translation is justified in section~\ref{sec-singlerod}.

\begin{table}
  \begin{center}
\begin{tabular}{|c|c|c|}
\hline
quantity & geometric notation & rod notation \\
\hline
quasi-velocities & $\xi^i$ & $(\c v, \c u)$\\
generalized momenta & $p_i$ & $(\c n, \c m)$\\
structure constants &$p_i \, c^i_{jk}$ &  $(-\c n^\times,-\c m^\times)$\\
Lagrangian function& $\mathcal L$ & energy-density: $W$ + interaction terms.\\
\hline
\end{tabular}
\caption{Correspondence table between geometric objects and quantities pertinent for rod theories.\label{tab-translate}}
\end{center}
\end{table}

\section{Framed Hamiltonian treatment of variational problems for rods\label{sec-theory}}

Before the construction begins, we fix the stage. Sections~\ref{sec-hamcla}
and~\ref{sec-geomvar} recall classical material: the coordinate route to the
Hamiltonian picture, and the characterization of critical curves of a 1-form
as the vortex lines of its differential~\cite{ar00}. 
Sections~\ref{sec-extTan} and~\ref{sec-frame} then assemble the framed
extended tangent bundle from three standard ingredients: the left
trivialization of the tangent bundle of a Lie group~\cite{mara99}, the
quasi-velocities of Hamel's formalism as coordinates on the
fibres~\cite{blma09}, and the contact forms of the jet
space~\cite{kr15,gri83}. The assembly is elementary but must be proved
somewhere; the framing statements are collected in
appendix~\ref{app-framing}. The construction proper begins in
section~\ref{sec-lepagecartan}.

\subsection{Calculus of variation: the classical Hamiltonian picture\label{sec-hamcla}} 

Start with a functional of the form
\be{varClassical}
{\mathcal E}[q]=\int_a^b \mathcal L\left(q^1(s),\cdots,q^n(s),\dot q^1(s),\cdots, \dot q^n(s),\, s\right) ~ds,
\ee
where $(q^1,\cdots,q^n)\in\mathbb R^n$ and $\dot {()}$ denotes derivative with respect to $s$.

The aim of the game is then to find a critical function $q:[a,b]\subset\mathbb R\to \mathbb R^n$ for~\eqref{varClassical}. That is, we look for those special curves $q$ such that when moved, $\mathcal E[q]$ is left unchanged at first order. Technically, we ask that $\delta {\mathcal E}|_q[\tau]:=\sd{}{\varepsilon} \mathcal E[q+\varepsilon\,\tau]|_{\varepsilon =0}$ vanishes for all acceptable perturbation $\tau(s)=\big(\tau^1(s),\cdots,\tau^n(s)\big)$. For fixed boundaries, that is equivalent to the Euler-Lagrange equations
\be{ELclassical}
\forall i:\qquad\qquad \pd{\mathcal L}{q^i} - \sd{}{s} \pd{\mathcal L}{\dot q^i} =0.
\ee

For free boundaries, $\delta {\mathcal E}=0$ is equivalent to both\re{ELclassical} and 
\begin{align}\label{freeBC}
\left .\pd{\mathcal L}{\dot q^i}\right|_{\big(q(a),\dot q(a),a\big)} =& 0,&\left .\pd{\mathcal L}{\dot q^i}\right|_{\big(q(b),\dot q(b),b\big)} =& 0.
\end{align}

Classically, to build the Hamiltonian form of the Euler-Lagrange equations, we first define the generalized momenta 
\begin{align}\label{ClassicalMomenta}
p_i = \partial_{\dot q^i}\mathcal L.
\end{align}
After checking the Legendre condition that the matrix $\partial_{\dot q^i\dot q^j} \mathcal L$ is positive definite, we can then in principle express the velocities $\dot q^i(q,p)$ as functions of the positions and momenta.

We then perform a Legendre transform defining 
$$H=p_i\, \dot q^i(q,p)- \mathcal L\bigl(q,\dot q(q,p),s\bigr),$$
and observe that the system~\eqref{ELclassical}  of $n$ second-order differential equations is equivalent to the following system of $2n$ first-order differential equations
\begin{align}
\dot q^i&=\partial_{p_i} H,&
\dot p^i&=-\partial_{q^i} H.
\end{align}

Geometers note that $(q^1,\cdots,q^n)$ are a system of coordinates on some manifold $Q$ and casts $H$ as a function on the (extended) cotangent bundle of $Q$ where one can define the canonical symplectic form $\omega^2=dp_i\wedge dq^i$ and a Poisson bracket. Both objects yield a rich structure. In particular the invariants of the Hamiltonian flow are easily studied and connected via N\oe{}ther's theorems with the continuous one-parameter groups of symmetries of the system. 

Note that because this situation often arises in the case of dynamical problems in which $\mathcal L$ is the Lagrangian of some system and $\mathcal E$ its action integral, it is standard to only look for extremal curves and not for  minimizing curves of $\mathcal E$. However, we intend to study the elastic energy stored in a rod. In that case, equilibria are defined as extremal curves of this elastic energy but \emph{stable} equilibria are local \emph{minimizers} of the energy. We would therefore like to have tools which are also adapted to the study of the second variation of $\mathcal E$.

\subsection{A geometric variational problem \label{sec-geomvar}} 

A different problem mentioned for example in~\cite{ar00} consists in studying the integral of a 1-form $\omega$ on some manifold $M$ over a certain set of possible curves $\gamma$ with images in $M$.
\begin{align}
\label{funGeom}
\mathscr E =\int_\gamma \omega
\end{align}

We could then ask to find a curve that locally minimizes $\mathscr E$. To encode the idea of perturbing $\gamma$, we can then ask how $\mathscr E$ would change if we let $\gamma$ follow the flow associated with some vector field $W$ defined on $M$. If it is so that $\gamma$ is injective (no self-intersection) and $\mathscr E$ would increase for all such flows, then $\gamma$ is indeed a local minimizer of $\mathscr E$.

The first variation of~\eqref{funGeom} is then defined as the following operator on vector fields defined on an open neighbourhood of $\gamma$: 
\begin{align*}
\delta \mathscr E[W]=\left.\sd{}{\varepsilon} \int_{\phi_W(\varepsilon,\gamma)} \omega \right|_{\varepsilon=0},
\end{align*}
where $\phi_W$ denote the flow associated with the vector field $W$ (an introduction can be found in~\cite{Lee12}).

The Stokes theorem applied to forms shows (see appendix~\ref{app-varform}) that 
\begin{align}
\label{varGeom}
\delta \mathscr E[W] =\left[\iota_W \omega\right]_a^b +\int_\gamma  \iota_W d\omega
\end{align}
where $\iota_W$ is the contraction with $W$.

The first term in~\eqref{varGeom} plays the same role as the boundary terms appearing in the integration by part that yields the Euler-Lagrange equations in the calculus version of the argument. When minimizing with fixed boundaries, we ask that $W$ vanishes at $\gamma(a)$ and $\gamma(b)$ -- that is the end points do not flow. In that case the boundary term in~\eqref{varGeom} vanishes identically. When minimizing with free boundaries, the curve $\gamma$ must start and end at points of $M$ on which the form $\omega$ vanishes.

To study whether the second term of~\eqref{varGeom} identically vanishes, we focus on the particular case for which the following two conditions hold: $(i)$ the manifold $M$ is odd-dimensional and $(ii)$ the form $d\omega$ is non-degenerate (i.e. its null space is of dimension 1). When both those conditions hold the null space of a non-degenerate two-forms at a particular point of $M$ defines a one-dimensional vector subspace of the tangent space. The collection of those directions at all points of $M$ define a slope field on $M$. Curves which are tangent to that slope field are called vortex lines of the 2-form (see~\cite{ar00} also summarized in appendix~\ref{app-vortex}). One can show assuming both conditions $(i)$ and $(ii)$, that $\delta \mathscr E = 0$ iff the curve $\gamma$ is along the null direction of $d\omega$ (see appendix~\ref{app-vortex}); that is $\gamma$ is a vortex line of $d\omega$. This yields an encoding of $\gamma$ based on the integration of a system of first order equations (up to a parametrization). That is to say $\gamma$ is a flow curve of a vector field that is defined algebraically which is reminiscent of the Hamiltonian form of the Euler-Lagrange equations.

\subsection{The extended tangent-bundle $\overline{TQ}$\label{sec-extTan}}

Two problems arise when trying to bridge the gap between the last two sections.  First, 
expression~\eqref{varClassical} is manifestly not the integral of a 1-form on~$Q$. This will be addressed here by casting $\mathcal E$ as an integral along a curve in the extended tangent bundle $\overline{TQ}:=TQ\times[a,b]$ instead of an integral along a curve in $Q$. In that bigger space, $ds$ is a form and $\mathcal L$ is a function. However, the second problem arises: the differential  of $\mathcal L ds$ is a (badly) singular two-form so that the argument about vortex lines of section~\ref{sec-geomvar} and appendix~\ref{app-vortex} does not apply. That will be addressed in section~\ref{sec-lepagecartan}.

Starting with\re{varClassical}, assume that the $q^i$ form a set of coordinates on some manifold $Q$ of dimension $n$. Then the functions $q^i:[a,b]\subset \mathbb R\to \mathbb R$ appearing in~\eqref{varClassical} encode a curve $\gamma$ with image in $Q$:
\begin{align*}
\gamma:[a,b]\subset\mathbb R\to Q: s\mapsto \gamma(s),
\end{align*}
where $\gamma(s)$ has coordinates $\left(q^1(s),\cdots,q^n(s)\right)$. 

\begin{defi}  Let $TQ$ be the tangent bundle of a manifold~$Q$, and $[a,b]\subset \mathbb R$ a real interval. We define the \textbf{extended space} $\overline Q=Q\times \mathbb [a,b]$ and the \textbf{extended tangent-bundle}\footnote{Readers from geometry will recognize $\overline{TQ}$ as the jet space $J^1([a,b],Q)$.}\footnote{Note that $\overline{TQ}\neq T\overline Q$. The former space is a manifold of dimension $2 n+1$ and the latter (which we will not use here) is a manifold of dimension $2n+2$ where $n = \textrm{dim } Q$.} $\overline{TQ}=TQ \times [a,b]$.
\end{defi}

 The space $\overline{TQ}$ therefore consists of elements noted $(q,X,s)$ which contain three pieces of information: a point $q\in Q$, a vector $X\in T_qQ$ tangent to $Q$ at $q$ and a real number $s\in[a,b]$. This space also forms a vector bundle of rank $n$ over $\overline Q$ with the associated submersion $\pi:\overline{TQ}\mapsto\overline Q:(q,X,s)\mapsto (q,s)$.  The fibre $\pi^{-1}(q,s)$ of $\overline{TQ}$ over $(q,s)$ is  the tangent space to $Q$ at $q$: $T_qQ$. We will also make use of the submersion $\sigma:\overline{TQ}\to Q:(q,X,s)\mapsto q$ to $Q$ and its push-forward\footnote{That is an operator on vectors tangent to $\overline{TQ}$ that yield the projected vector tangent to $Q$. It is a 1-form with values in the tangent space of $Q$ at $q$. Sometimes one also notes $d\sigma = \sigma_\star$.}~$\sigma_\star$. That is if $C(t) = (q(t),X(t),s(t))$ is a general curve on $\overline {TQ}$, we simply have $\sigma_\star(C'(t)) = q'(t)\in T_qQ.$ Finally, recall that in general if $\zeta:A\to B$ is a submersion between manifold, we say that a vector $X_a$ tangent to $A$ is \textbf{vertical} for $\zeta$ whenever $\zeta_\star(X_a)= 0$.
 
 \begin{defi} \label{def-lift}Given a smooth curve $\gamma:[a,b]\to Q: s\mapsto \gamma(s)$, the lifted curve $L[\gamma]$ is a curve with image in $\overline{TQ}$ defined by 
\begin{align} \label{deflift}
L[\gamma]:[a,b]\to \overline{TQ}: s \mapsto \big (\gamma(s),\, \gamma'(s),\, s\big).\end{align}
\end{defi}
 \vspace{0mm}
 
 Because the function $\mathcal L$ appearing in~\eqref{varClassical} depends on the $\dot q^i$, that is on the tangent vector to $\gamma$, it is not a function on $Q =\mathbb R^n$. It is however a function on the extended tangent-bundle $\overline{TQ}$. We can therefore re-write\re{varClassical} as
 \be{varClassicalGeom}
 {\mathcal E}:=\int_a^b \mathcal L(q^i,\dot q^i,s) ds = \int_{L[\gamma]} \mathcal L \, ds =:\mathscr E\bigl[L[\gamma]\bigr],
 \ee
 which is indeed the integral of a 1-form on $\overline{TQ}$. In what follows, the manifold $\overline{TQ}$ constructed above will play the role of the manifold $M$ of section~\ref{sec-geomvar}.
 
However, locally minimizing~\eqref{varClassical} over neighbouring curves in $Q$ is not equivalent to locally minimizing $\mathscr E[C]$ in~\eqref{varClassicalGeom} over neighbouring curves in $\overline{TQ}$. We need to minimize over those curves in $\overline{TQ}$ that are (a~reparametrization of) a lift of a curve in $Q$. The following question therefore needs an answer. Given a general curve $C$ with image in $\overline{TQ}$, does a curve of $\gamma\subset Q$ exists such that it lifts to $C$?

\begin{defi} A curve $C :t\in [A,B] \mapsto C(t)\in\bar{TQ}$ is \textbf{holonomic} iff there exists a curve $c : \tilde s\in [a,b] \mapsto c(\tilde s)\in Q$ and a monotonic and smooth mapping $s:[A,B]\subset\mathbb R\to [a,b]\subset\mathbb R$ such that 
\be{geomlift}
C (t) =\left(\left .c\right|_{s(t)}, \, \left. c'\right|_{s(t)} ,\, s(t) \right).
\ee
\end{defi}

The following lemma handles the characterization whether a specific curve $C$ is holonomic.
\begin{lem}\label{lem-integrable}
A smooth curve $C :t\in [A,B] \mapsto C(t)\in\bar{TQ}$ is \emph{holonomic}  \emph{iff} 
\begin{align}
\forall t\in[A,B]: (\sigma_\star -X ds) (C'(t))=0 \textrm{ and } ds(C'(t))\neq 0.
\end{align}
\end{lem}
\begin{proof} 
The condition is necessary because by definition, a holonomic curve $C$ admits a curve $c$ with image in $Q$ and a monotonic function $s$ respecting~\eqref{geomlift}. Accordingly,
\begin{align*}
(\sigma_\star -X ds) (C'(t))= \sd{}{t} c|_{s(t)} - c'|_{s(t)} s'(t) = 0,
\end{align*}
where the last equality holds because of the chain rule. Furthermore, we have $ds(C'(t)) =s'(t) \neq 0$ since $s$ is monotonic.

The condition is sufficient because if a curve $C(t) = (q(t), X(t),s(t))$ is smooth, then the condition $ds(C')\neq 0$ implies that $s(t)$ is strictly monotonic, and we can use it as the function of the same name appearing in the definition of holonomic curves. Because $s$ is smooth and strictly monotonic, it can be inverted. Noting, the inverse $t(s)$, we can define the curve $c(s) = q(t(s))$ with image in $Q$. All that remains is to prove that $X(t)= c'(s(t))$. But this is immediate since the condition
\begin{align*}
(\sigma_\star -X ds) (C'(t)) = q'(t) - X(t) s'(t) = 0.
\end{align*}
yields $X(t) = q'(t) /s'(t) =q'(t) t'(s) = \sd{}{s} q(t(s)).$
\end{proof}

Note in particular that the second condition in the lemma, namely $ds(C'(t))\neq 0$ is important to exclude curves on $\overline{TQ}$ which are vertical for the submersion $\pi$; those are curves that lie within a single fibre over a certain $(q,s)\in\overline{Q}$.

\subsection{Framing $\overline{TQ}$\label{sec-frame}}

Classically when working with a manifold $M$ of dimension $n$, one defines $n$ coordinates $x^1,\cdots, x^n$ on (patches of) the manifold. Technically, the $x^i$ are $n$ sufficiently smooth functions defined on $M$. Coordinates functions are such that,
\begin{itemize}
  \item specifying values for all $n$ functions is equivalent to pinpointing a specific element $m\in M$, 
  \item their differentials provide bases for the cotangent spaces at every point of $M$, and 
  \item the partial derivatives $\partial_{x^i}$ provides a basis of the tangent spaces at every $m\in M$.
\end{itemize} Accordingly, a point $m\in M$, a vector $X\in T_m M$ and a cotangent vector  $\alpha\in T^\star_m M$ can each be encoded with $n$ real numbers, respectively the coordinates $m^i$ of $m$, the components $X^i$ such that $X=X^i \partial_{x^i}$ and the components $\alpha_i$ such that $\alpha = \alpha_i \, dx^i.$ So coordinates  actually achieve three encodings all at once.

When the manifold is a Lie group, there exists another possible tack. We can achieve the three objectives listed above without coordinates by using frames. A global frame of an $n$-dimensional manifold is a set of $n$ continuous vector fields on $M$ such that for every $m\in M$, the vector fields evaluated at $m$ provide a basis for $T_mM$. There exist manifolds that can not be globally framed; $S^2$ is a notorious example. But for Lie groups, any basis of the algebra -- that is the set of left-invariant vector fields -- provides such a frame. The dual algebra also provides local bases for covectors. For Lie groups of dimension $n\in\mathbb N$ that cannot be covered by a single chart of coordinates, it is always possible to define a set of $k$ functions with $k>n$ such that the values of those $k$ functions determine a unique point in $M$. Naturally those values cannot be chosen in all generality. We have a mapping $M\to\mathbb R^k$ that is injective but not surjective. With such a construction, we can therefore compute for a given problem by specifying $k+n+n$ real numbers where the first term accounts for the real values needed to specify a point $m\in M$, the second (resp. third) provides components for a cotangent (resp. tangent) vector at $m$.

However, the previous section defined a variational problem on $\overline{TQ}$ and not on the Lie group $Q$. Hence $\overline{TQ}$ needs its own framing. With $Q$ a Lie group of dimension $n\in\mathbb N$, any basis $(b_1,
\cdots,b_n)$ of the tangent space $T_eQ$ at the identity $e\in Q$ extends uniquely to a $n$-tuple $(B_1,\cdots, B_n)$ of left-invariant vector fields that are smooth and everywhere linearly independent. We also define the dual basis $(\beta^1,\cdots,\beta^n)$ of covectors on $Q$. To move to $\overline{TQ}$ recall the trivial projection $\sigma: \overline{TQ}\to Q$ and define $\theta^i = \sigma^\star \beta^i$ and the functions $\xi^i = \theta^i(X)$ where any element of $\overline{TQ}$ can be written as $(q,X,s)$ with $q\in Q$, $X\in T_q Q$ and $s\in [a,b]$. The functions $\xi^i$ are to frames what the quasi-velocities $\dot q^i$ are to coordinate systems.

In this way $(\theta^1,\cdots ,\theta^n,d{\xi^1},\cdots, d{\xi^n}, ds)$ provide a coframing of $\overline{TQ}$ (see appendix~\ref{app-framing}). Finally, note $\{E_{q^1},\cdots, E_{q^n}, E_{{\xi^1}}, \cdots , E_{{\xi^n}}, E_s\}$ the dual frame of $\overline{TQ}$. A general vector field $W$ on $\overline{TQ}$ can therefore be expanded as 
\begin{equation} 
  W = W_{q}^i\, E_{q^i} + W_\xi^i\, E_{\xi^i} + W_s\, E_s.
\end{equation}

Finally, recall from the appendix that for a specific Lie-group and a specific choice of left-invariant frame there exist $n^3$ constants $c^i_{jk}$ called the structure constant and such that 
\begin{align} 
  [B_i,B_j]& = c^k_{ij} \, B_k, &
  d\beta^k &=-\frac1 2\, c^k_{ij}\beta^i\wedge \beta^j,
\end{align}
where the square brackets denote the commutator.

This yields a commutator structure for the frame of $\overline{TQ}$ developed in~\eqref{appEcomm} and reproduced here for convenience: 
\begin{align}\label{Ecomm}
  \left[E_{q^i},E_{q^j}\right] &= c^k_{ij} \,E_{q^k}, &
    \left[E_{q^i},E_{\xi^j}\right] &=0, &
    \left[E_{\xi^i},E_{\xi^j}\right] &=0, &
    \left[E_{q^i},E_s\right] &=0, &
    \left[E_{\xi^i},E_{s}\right] &=0.
\end{align}

Also from the appendix equation~\eqref{appdthetafin} states that for any two vector fields $X$ and $W$ on $\overline{TQ}$:
\begin{equation}\label{dthetafin}
  d\theta^i(X,W) = - c^i_{jk} X_q^j\, W_q^k.
\end{equation}

\begin{table}
  \begin{center}
\begin{tabular}{c||c}
  Variables & Definitions\\
  \hline
  \hline
  $(b_1,\cdots,b_n)$ & A basis of $T_e Q$.\\
  $(B_1,\cdots, B_n)$ & Left-invariant extensions of the $b_i$\\
  $(\beta^1,\cdots, \beta^n)$ & Everywhere dual to the $B_i$\\
  $\pi:\overline{TQ}\to Q$ & Trivial projection of the (extended) vector bundle on $Q$.\\
$\theta^i$ & Pull-backs of the $\beta^i$ under $\pi$\\
$\xi^i|_{q,X,s}:= \theta^i(X)$ & Functions on $\overline{TQ}$\\
$(E_{q^1},\cdots, E_{q^n}, E_{{\xi^1}},\cdots,E_{{\xi^n}},E_s)$ & A frame of $\overline{TQ}$ that is everywhere dual to $(\theta^1,\cdots,\theta^n,d{\xi^1},\cdot,d{\xi^n},ds)$
\end{tabular}
\caption{A table of notations regarding bases on the tangent and cotangent spaces of $Q$ and $\overline{TQ}$ (the extended tangent bundle of $Q$).}
\end{center}
\end{table}

\subsection{Differential forms on $\overline{TQ}$ with holonomic vortex lines: Cartan and Lepage\label{sec-lepagecartan}} 

Although the definition~\eqref{varClassicalGeom} of $\mathscr E[\Gamma]$ is an obvious candidate for a functional $\mathcal F$ on $\overline{TQ}$ such that
$\mathcal F[L[\gamma]] = \mathcal E[\gamma]$ for all smooth curves $\gamma$ on $Q$,
it is not unique. This freedom turns out to be useful.

First note that for any choice of $n$ functions $p_1,\dots,p_n$ on $\overline{TQ}$,
the 1-form
\be{defTh}
\Theta = \mathcal L\,ds + p_i\,(\theta^i - \xi^i\,ds)
\ee
is such that for any holonomic curve $\Gamma$,
\be{LepageCongruence}
{E}[\Gamma] \equiv \int_\Gamma \Theta = \int_\Gamma \mathcal L\,ds = \mathscr E[\Gamma].
\ee
The central equality holds because lemma~\ref{lem-integrable} states that for holonomic curves, we have $(\theta^i-\xi^i ds)(\Gamma')  = 0$ since $\theta^i - \xi^i ds$ are the components of the covector valued form $\sigma_\star - Xds$.

What follows is a distillation, to our one-dimensional problem, of a particularly nice
chain of ideas in geometric mechanics. Cartan~\cite{ca22} showed that the form $p_i\,dq^i - H\,dt$ is a relative integral invariant of the
Hamiltonian flow on the extended phase space. Lepage~\cite{le36}, seeking to
extend these ideas to field theories---where our arc-length integral becomes a
multiple integral over some base space---noticed that $\mathcal L\,ds$ and the form
$\Theta$ of~\eqref{defTh} give the same integral on holonomic surfaces (holonomic
curves, for us). This was later developed, in the language of jet spaces, into the
variational sequence by Krupka~\cite{kr73,kr15}, who coined the term
\emph{Lepage equivalent} for two forms so related.

It may be precisely because the forms-modulo-contact idea was developed for multiple
integrals that it never became a standard way of introducing the Hamiltonian
construction. Yet in one dimension too it clarifies the passage from the Lagrangian
to the Hamiltonian formalism: we will see that $(i)$ the functions $p_i$ appearing in
$\Theta$ ought to be the conjugate momenta, so their appearance is now motivated rather
than postulated; $(ii)$ when the $p_i$ are the conjugate momenta, $\Theta$ is the
Poincaré--Cartan form; and $(iii)$ its integral is therefore exactly the energy of the
rod, provided we integrate over holonomic curves---the only curves on $\overline{TQ}$
that correspond to actual physical rods.

In a framed extended tangent bundle, the precise statement is the following.
\begin{proposition}\label{prop-pforced} Let $\overline{TQ}$ be an extended tangent bundle framed as in section~\ref{sec-frame} and $\mathcal L$ a function defined on $\overline{TQ}$ such that the $n\times n$ matrix of elements $P_{ij} = E_{\xi^i}E_{\xi^j} \mathcal L$ is invertible. Let also $p_1,\cdots p_n$ be functions on $\overline{TQ}$ and $\Theta$ the form defined in~\eqref{defTh}.

The 2-form $d\Theta$ is nonsingular and the vortex lines of $\Theta$ are holonomic if
and only if
\[
\forall i \in \{1,\dots,n\}:\quad p_i = E_{\xi^i}\mathcal L.
\]
\end{proposition}
\begin{proof}
  First, with $X$ and $W$ vector fields on $\overline{TQ}$ and written in component on the frame defined in section~\ref{sec-frame} $X = X_q^i\, E_{q^i} + X_\xi^i \,E_{\xi^i} + X_s\, E_s$ and $W = W_q^i\, E_{q^i} + W_\xi^i \,E_{\xi^i} + W_s\, E_s$, compute
\begin{eqnarray}
d\Theta (X,\, W) &=& d (\mathcal L ds) (X,\, W) +\nonumber \\
&&\quad\Big(\big( E_{q^j} (p_i) \,\theta^j+ E_{\xi^j}  (p_i)\,  d\xi^j + E_s (p_i) \, ds\big ) \wedge( \theta^i-\xi^i \, ds) + p_i \, d\theta^i- p_i\, d\xi^i \wedge ds\Big)(X,\, W) \nonumber\\ 
&=& W_s \, \Big(E_{q^i} \mathcal L \, X_q^i + E_{\xi^j} \mathcal L \, X_\xi^j 
- \xi^i\big( (E_{q^j} p_i) X_q^j +(E_{\xi^j} p_i) X_\xi^j  \big)
-E_s p_i X_q^i
 - p_i\, X_\xi^i \Big)\nonumber\\ 
&&\quad + W_q^i \, \Big(-(E_{q^i} \mathcal L )  X_s  + (E_{q^j} p_i) X_q^j +(E_{\xi^j} p_i) X_\xi^j  + (E_s p_i)\, X_s \\
&&\quad\quad  -  E_{q^i} (p_j) X_q^j +X_s \xi^j  E_{q^i} p_j - p_k\, c^k_{ji} X_q^j\Big)\nonumber\\
&&\quad + W_\xi^i \Big( -(E_{\xi^i} \mathcal L )  X_s  - E_{\xi^i}(p_j) X_q^j+\xi^j X_s\, E_{\xi^i} p_j + p_i \,X_s\Big). \nonumber
\end{eqnarray}
Hence finding $X$ such that $\forall W\in\mathfrak X(\bar{TQ}): d\Theta(X,W) = 0$ is equivalent to solving
\begin{eqnarray}
\Big( E_{q^i} \mathcal L- \xi^j\,  (E_{q^i} p_j) -E_s p_i \,  \Big)   X_q^i 
+ \Big(E_{\xi^i} \mathcal L- \xi^j (E_{\xi^i} p_j) - p_i \Big) X_\xi^i 
&=&0\nonumber\\ 
\Big((E_{q^i} p_j) -  E_{q^j} (p_i) - p_k c^k_{ij}\Big) X_q^i 
+ E_{\xi^i} p_j~X_\xi^i
+\Big( E_s p_j\, - E_{q^j} \mathcal L +  \xi^i  E_{q^j} p_i\Big)  X_s  &=&0\label{sysdeff}\\
 - E_{\xi^j}(p_i) X_q^i + \Big( -(E_{\xi^j} \mathcal L )  +\xi^i \, E_{\xi^j} p_i + p_j\Big) \,X_s&=&0.\nonumber
\end{eqnarray}

\textbf{Necessity of $p_i =E_{\xi^i}\mathcal L$.} 
Suppose $d\Theta$ is nonsingular with holonomic vortex lines. The tangent vector $X$ to the vortex line must therefore respect $ds(X)\neq 0$ and $(\theta^i -\xi^i ds)(X) = 0$, that is $X_s\neq 0$ and $X_q^i = \xi^i\, X_s$. Substituting these relations into the last equation of\re{sysdeff} implies 
\be{fdef}
p_i = E_{\xi^i} \mathcal L.
\ee 

\textbf{Sufficiency of $p_i = E_{\xi^i}\mathcal L$.} Assuming\re{fdef} and substituting in\re{sysdeff} yields 
\begin{eqnarray}
\Big( E_{q^i} \mathcal L- \xi^j\,  (E_{q^i} {E_{\xi^j} \mathcal L})  - E_s E_{\xi^i} \mathcal L \Big)   X_q^i 
- X_\xi^i ~~ P_{ij}~\xi^j
&=&0,\nonumber\\
\Big((E_{q^i} {E_{\xi^j} \mathcal L}) -  E_{q^j} ({E_{\xi^i} \mathcal L})- p_k c^k_{ij}\Big) X_q^i 
+ X_\xi^i~ P_{ij}
+\Big( E_s {E_{\xi^j} \mathcal L}\, - E_{q^j} \mathcal L+  \xi^i  E_{q^j} E_{\xi^i}\mathcal L \Big)  X_s  &=&0,\label{sysdeff2}\\
 P_{ji} (X_q^i - \xi^i  \,X_s)&=&0.\nonumber
\end{eqnarray}

Since $P$ is assumed invertible,  the last equation of\re{sysdeff2} implies that 
\begin{equation}\label{prooffirstcond}
  X_q^i= \xi^i \, X_s.
\end{equation}
 Substituting these relation in the first two lines of\eqref{sysdeff2} gives 
\be{sysdeff3}
\begin{split}
\xi^j (X_\xi^i\,P_{ij} )&= \xi^j\Big( E_{q^j} \mathcal L- \xi^i\,  (E_{q^j} {E_{\xi^i} \mathcal L})  
- \,E_s {E_{\xi^j} \mathcal L} \Big)\, X_s, \\
 X_\xi^i P_{ij} &= \Big( E_{q^j} \mathcal L- E_s {E_{\xi^j} \mathcal L}-\xi^i E_{q^i} {E_{\xi^j} \mathcal L} +\xi^i\, p_k\, c^{k}_{ij}\Big)  X_s.
\end{split} 
\ee

Therefore, if $X$ is non-trivial and in the null space of $d\Theta$, it must have $X_s\neq 0$ because otherwise $X= 0$. This together with~\eqref{prooffirstcond} shows that any vector in the kernel of $d\Theta$ is holonomic.

Finally, the second set of equations in~\eqref{sysdeff3} implies the first. Gathering we therefore have that if we choose the functions $p_i = E_{\xi^i} \mathcal L$, the linear system~\eqref{sysdeff} is equivalent to the linear system
\begin{equation} \label{forward}
  \begin{cases} 
     X_q^i&= \xi^i \, X_s,\\
      X_\xi^i P_{ij} &= \Big( E_{q^j} \mathcal L- E_s {E_{\xi^j} \mathcal L}-\xi^i E_{q^i} {E_{\xi^j} \mathcal L} +\xi^i\, p_k\, c^{k}_{ij}\Big)  X_s.
  \end{cases}
\end{equation}
The solution space of~\eqref{forward} is one-dimensional (parametrized for instance by $X_s$). That is because the matrix $P_{ij}$ is invertible. Accordingly, the null-space of $d\Theta$ is one-dimensional: $d\Theta$ is non-singular 2-form on the odd-dimensional space $\overline{TQ}$.

\end{proof}

\begin{rem} 
  The second set of equations in\re{sysdeff3} can be rewritten as
\be{ELgen}
X_s \, E_{q^j} \mathcal L-  \Big( X_s \, E_s +X_q^i\, E_{q^i} +X_\xi^i E_{\xi^i} \Big) ( {E_{\xi^j} \mathcal L}) +\xi^i\, p_k\, c^k_{ij}\, X_s= X_s \,  E_{q^j} \mathcal L-  X ( {E_{\xi^j} \mathcal L})+ \xi^i\, c^k_{ij}\, E_{\xi^k}(\mathcal L)\, X_s=0,
\ee
which, along a holonomic curve parametrized by $s$, are the equations first
written by Poincar\'e~\cite{poincare01} for left-invariant frames on Lie
groups and extended by Hamel~\cite{hamel04} to arbitrary frames --- the Hamel
equations, in the terminology of~\cite{blma09} --- and which reduce to the
Euler--Lagrange equations\re{ELclassical} when the frame is induced by
coordinates.
\end{rem}

\begin{rem}
Now that the $p_i$ are specified, we recognize the Poincaré--Cartan form since 
\[
\Theta = \mathcal L ds + p_i \, \left( \theta^i - \xi^i ds \right)  = p_i \, \theta^i - (p_i\, \xi^i - \mathcal L) ds  = p_i\, \theta^i -\mathcal H ds,
\]
where we recognized the Hamiltonian function $\mathcal H = p_i\, \xi^i -\mathcal L$.
\end{rem}

\subsection{Changing frame \label{sec-framechange}}

In the classical Hamiltonian picture, one needs to change variable from the pseudo-velocities ${\dot q^{i}}$ to the generalized momenta. It is in general a nonlinear change of variables and in many cases, computing the inverse transform is cumbersome if at all tractable. 

Here comes a first advantage of the frame approach. We can use the $p_i$ as new coordinates over the fibres (instead of the $\xi^i$). But it turns out that we never need to compute the change of variables explicitly. Instead, it is enough to compute the new coframe $(\theta^i,dp_i,ds)$ and its dual, call it $(D_{q^i},D_p^i,D_s)$. So translating expressions written in the old frames of section~\ref{sec-frame} to these new frames is always a linear operation. We completely do away with the need of inverting a nonlinear relation.

To proceed, note that because of the duality between our chosen frame and coframe, for any function $\varphi$ on $\overline{TQ}$, we have $d\varphi = (\theta^i\otimes E_{q^i} + d\xi^j\otimes E_{\xi^j} + ds\otimes E_s)\varphi$. In particular,
\[
d p_j = E_{q^i} p_j \, \theta^i + E_{\xi^k}  p_j\, d\xi^k + E_s  p_j\, ds = (E_{q^i}E_{\xi^j} \mathcal L)\, \theta^i+
(E_{\xi^k}E_{\xi^j} \mathcal L)\, d\xi^k + (E_s  E_{\xi^j} \mathcal L)\, ds,
\]
so that upon defining the matrix $Q_{ij} = E_{q^i}E_{\xi^j}\mathcal L$ and the line $\c t_j= E_s E_{\xi^j} \mathcal L$, the change of basis is given by 
\[
\bma
\theta^i & dp_j & ds 
\ema 
= 
\bma
\theta^a & d\xi^k& ds 
\ema 
\underbrace{
\bma
\Big(\textrm{Id}\Big)_{a}^i & \Big( Q\Big)_{kj} & 0\\
0 & \Big( P\Big)_{kj}  & 0 \\
0& \c t_j & 1
\ema
}_A.
\]

Let $\{D_{q^i}, \, D_{p}^j,\, D_s\}$ be the basis dual to $\{\theta^i, dp^j,ds\}$, there exists a matrix $B:\bar{TQ}\mapsto\mathbb R^{(2n+1)\times(2n+1)}$ such that 
\be{computeDframe}
\bma 
D_q\\
D_p\\
D_s\ema = 
B \bma E_q\\ E_\xi\\ E_s\ema = \bma \textrm{Id} & -Q P^{-1} & 0 \\ 0 & P^{-1} & 0 \\ 0 & - \c t P^{-1} & 1\ema \bma E_q\\ E_\xi\\ E_s\ema,
\ee
where $B= A^{-1}$ because of the duality. 

\begin{rem} When $Q\neq 0$ and $\c t \neq0$, the vectors $D_{q^i}\neq E_{q^i}$ and the vector $D_s\neq E_s$. 
\end{rem} 
\begin{rem} 
  However, the components are left unchanged. To see this, consider the components $X_q^i,\, X_\xi^j,\, X_s$ of a vector $X$ in the $E$ basis. Its components $Y_q^a,\,Y_{pb},\, Y_s$ in the $D$ basis are then computed by considering 
\[
X = \bma Y_q^a & Y_{pb}&  Y_s\ema D = \bma Y_q^a & Y_{pb}&  Y_s\ema B\, E ~~\Longrightarrow~~
\bma Y_q & Y_{p}&  Y_s\ema = \bma X_q & X_{\xi}&  X_s\ema A.
\]
In particular, we find $Y_q = X_q$ and $Y_s = X_s$ while $Y_p = X_q \, Q + X_\xi \, P + X_s \, \c t$. To simplify notations we rename the column $Y_p = X_p$ so that we have 
\[
X = X_q^i \, E_{q^i} + X_\xi^j\, E_{\xi^j} + X_s \, E_{s} =X_q^i \, D_{q^i} + X_{pj}\, D_{p}^j + X_s \, D_{s}.
\]
\end{rem}
\subsection{Hamiltonian picture on framed and extended phase space $\overline{TQ}$}

Now we consider the functional $E$ and write 
\[
E[\Gamma] = \int_\Gamma \Theta = \int_\Gamma \mathcal L ds + p_i\, (\theta^i - \xi^i ds) = \int_\Gamma \overbrace{(\mathcal L - p_i \, \xi^i )}^{-\h} ds + p_i\, \theta^i,
\]
where $\mathcal H$ is the Hamiltonian function of the system.
Asking that $\Gamma$ be critical for $E$ is equivalent to asking that it be a flow line of a vector field $X$ such that $\forall W: d\Theta (X,W) = 0$. We find
\begin{eqnarray}
d\Theta (X,W)&=& d(p_i \theta^i - \h ds) (X,W)\nonumber \\
&=&\big( dp_i\wedge \theta^i + p_i d\theta^i + ds\wedge d\h\big) (X,W)\nonumber \\
&=& X_{pi} W_q^i - X_q^i \, W_{pi}+p_i\, d\theta^i(X,W) + X_s\, W(\h) - W_s\, X(\h). \label{dTheta1}
\end{eqnarray}
Substituting\re{dthetafin} in\re{dTheta1} yields that $\forall W\in\mathfrak X(\bar{TQ})$,
\begin{eqnarray}
&d\Theta(X,W) =W_q^i \left(X_{pi} -  p_a \, c^a_{ji} X_q^j + X_s\, D_{q^i}(\h)\right) 
+ W_{pi} \left( X_s D_{p}^i(\h) - X_q^i\right)      +    W_s \left ( X_s\, D_s(\h) - X(\h) \right) =0,\nonumber \\
&\Longleftrightarrow\qquad\qquad \left\{
\begin{aligned}
X_{q}^i& = X_s\, D_p^i(\h),\\
X_{pi} &= - X_s\, D_{q^i}(\h) +  \, p_a\, c^a_{ji} X_q^j,\\
X_s D_s(\h) &= X(\h).
\end{aligned}
\right . &\label{HamiltonGen}
\end{eqnarray}
The first two equations are the standard Hamiltonian equations while the last is the property that $\h$ is  invariant under the dynamics whenever $\h$ has no explicit dependence on $s$.

\subsection{A Poisson structure on the extended tangent bundle}

To derive the Poisson structure associated with the Hamiltonian flow, we note that for any function $\varphi \in\mathfrak F(\bar{TQ})$, we have 
\begin{eqnarray}
X(\varphi) &=& X_{q}^i \, D_{q^i} \varphi + X_{pj} D_{p}^j \varphi   + X_s D_s\varphi\nonumber \\
&\stackrel{\eqref{HamiltonGen}}=&X_s\left( D_p^i(\h)\, D_{q^i} (\varphi) - D_{q^j}(\h) \, D_{p}^j (\varphi)  +
 p_a\, c^a_{ji} D_p^j(\h) D_p^i(\varphi) 
+ D_s (\varphi) \right) .
 \end{eqnarray}
 So that upon defining the bivector 
 \be{defJ}
 \mathcal J = D_{q^i}\otimes D_{p}^i - D_p^i\otimes D_{q^i}  + p_a \, c^a_{ij} D_p^j\otimes D_p^i,
 \ee
 and the poisson bracket 
 \be{defbra} 
 \{ F,G \} = \mathcal J(dF,dG),
 \ee
 We do find that 
 \[
 X(\varphi) = X_s \, \bigg( \big( \{ \varphi,\h\} \big) +D_s(\varphi) \bigg).
 \]
Accordingly when following a curve $\Gamma\subset \overline{TQ}$ that is a rod equilibrium, one sees any function $\varphi$ change at the following rate: 
 \begin{equation}\label{evolve}
 \sd{\varphi}s  = 
\{ \varphi,\h\}  +  D_s(\varphi).
 \end{equation} 

The bivector~\eqref{defJ} was read off the Hamiltonian
flow~\eqref{HamiltonGen}, so the bracket~\eqref{defbra} is bilinear,
antisymmetric and a derivation in each slot by construction; only the Jacobi
identity requires an argument. It holds because $\mathcal J$ is the inverse
of the restriction of the closed form $d\Theta$ to the subbundle
$\ker ds\subset T\overline{TQ}$: the bracket is therefore the one induced by
a closed nondegenerate two-form on each slice $s=\mathrm{const}$, for which
the Jacobi identity is classical~\cite{lima87}. 

Remembering that $X$ is by construction tangent to a holonomic curve, we must have that $(\theta^i-\xi^i ds)(X) = 0.$ That is $X_{q}^i = \xi^i \, X_s$. Substituting this relation in the first line of~\eqref{HamiltonGen}, immediately yields the quasi-velocities along an equilibrium: 
\begin{equation}\label{rDpH}
  \xi^i = D_{p_i} \mathcal H.
\end{equation}
That specifies how we move in $Q$ when following a rod equilibrium. However, when working in the Hamiltonian picture we usually have to reexpress these quasi-velocities as function of the generalized momenta and in general this requires an inversion of a nonlinear relation. When the inverted relation $\xi^i(q,p)$ is available and manageable, the
evolution in $Q$ is expressed directly as a function of the $p_i$. When it is
not --- because no closed form exists, or because the closed form is too
unwieldy to be of practical use --- we can write the more general system
\begin{equation}
  \begin{cases}
    q' &= B_i \, \xi^i,\\
    p_i'& = \{p_i, \mathcal H\} + D_s \mathcal H,\\
    {\xi^i}'  & =\{\xi^i,\mathcal H\} + D_s \xi^i. 
  \end{cases} 
\end{equation}

Remembering the algebra of section~\ref{sec-frame}, the last two equations can be written more explicitly as 
\begin{equation} \label{hamflow}
  \begin{cases} 
    q' & = B_i \, \xi^i,\\
    p_i' &= -D_{q^i} \mathcal H +c_{ji}^k \, p_k \xi^j,\\
    {\xi^i}' &= - \left(D_{q^j} \mathcal H + \left(p_a c^a_{jk} + Q_{kj}\right) \xi^k + \c t_j \right) \left(P^{-1}\right)^{ji}.
  \end{cases}
\end{equation}

 As is always the case on the (extended) tangent bundle, this Poisson structure~(\ref{defJ},\ref{defbra}) defined on $\overline{TQ}$ depends on $\mathcal H$. That is, if we change the constitutive laws of the rod (i.e., we change $\mathcal H$), then the definitions of $p_i$ change and the Poisson bracket adapts.

 Most treatments pass to the cotangent bundle $\overline{T^\star Q}$, where the
symplectic structure is canonical and problem-independent. In this study, we deliberately stay on the
extended tangent bundle: the construction is self-contained here, the momenta arise in
the frame without a fibre map, and---as we will see---the freedom to frame the fibre as
we please is specifically useful to distinguish the conjugate momenta from the internal stresses of the rod (see section~\ref{sec-interact}).

There is however a reason to keep both pictures in view. A constraint such as
unshearability or inextensibility may be imposed by penalizing the offending strain with
a stiffness $A$ and letting $A\to\infty$. On the Lagrangian side this limit is singular,
as it removes a degree of freedom; but the conjugate momentum stays finite (it is the
constraint reaction), and the stiffness enters the Hamiltonian only through $1/A$, so on
the cotangent bundle---where the bracket carries no $A$ at all---the limit $1/A\to0$ is
regular. The penalty thus selects the physically correct (d'Alembert) constrained
dynamics while remaining within the smooth Hamiltonian framework. We develop this limit,
and the constraint-adapted choices of frame it suggests, in a companion paper; here we
note only that it is the cotangent counterpart of the present construction that makes it
transparent.

Finally note that although the definition~\eqref{defJ} was made using our frames. The resulting bivector is a geometric object the definition and existence of which is completely independent of any choice of coordinates or frames. It does project to the canonical bracket under the standard fibre projection. The third term in~\eqref{defJ} is precisely the price we pay for choosing to work with the left-invariant forms $\beta^i$ which are adapted to our Lie group setting but which are not closed (hence the structure constants). Reassuringly, we will see in the applications that their appearance expresses a well known trick in rod theories: when working in the local basis $\tfrac d {ds}\to \tfrac d{ds} + \c u^\times.$ The structure constant for $SO(3)$ exactly encode the move to the material frame.
 
 \subsection{The story so far}
 We started by wanting to minimize
 $$\int_\Gamma \mathcal L ds,$$
 over holonomic curves on $\overline{TQ}$ and assuming certain boundary conditions. We then showed that for holonomic curves $\Gamma$ on $\overline{TQ}$, the functional is equal to
 \begin{align}\label{ThetaSummary}
 \int_\Gamma \mathcal L ds= \int_\Gamma p_i\, \theta^i - \mathcal H ds,
 \end{align}
where $p_i = E_{\xi^i}\mathcal L$ and $\mathcal H = p_i\, \xi^i-\mathcal L$ is a function on $\overline{TQ}$. But this second functional has the advantage that its critical curves are vortex lines of $d(p_i\theta^i -\mathcal H ds)$.
 
 Therefore, finding extremal curves of $\int_\Gamma \mathcal L ds$ over holonomic curves is equivalent to finding extremal curves of $\int_\Gamma p_i\, \theta^i -H ds$ over \emph{all smooth curves} in $\overline{TQ}$ (since we are guaranteed that its vortex lines are holonomic).

 The second integral in equation~\eqref{ThetaSummary} has the advantage that it readily defines critical curves as the vortex lines of $d(p_i\theta^i-\mathcal H ds)$. These vortex lines are found to obey the set of equations~\eqref{HamiltonGen} reproduced here for convenience: 
 \begin{equation*}
 \left\{
\begin{aligned}
X_{q}^i& = X_s\, D_p^i(\h),\\
X_{pi} &= - X_s\, D_{q^i}(\h) +  \, p_a\, c^a_{ji} X_q^j,\\
X_s D_s(\h) &= X(\h).
\end{aligned}
\right .
 \end{equation*}

Equation~\eqref{ThetaSummary} hints at the physical meaning of $p_i\,\theta^i|_{\gamma(s^\star)}$ as the generalized force applied by the material described by $s>s^\star$ on the material labeled by $s\leq s^\star$. Indeed, if the rod is along a critical curve (i.e.: a mechanical equilibrium), then we reproduce the question of isolating the portion that is between the labels $(c,d)$ and consider the variation 
 \begin{align*}
 \delta \left(\int_{\Gamma_c^d} p_i\theta^i-\mathcal H ds\right) [W] = \int_{\Gamma_c^d} \iota_W d(p_i\theta^i-\mathcal H ds)  + \int_{\Gamma_c^d}d (\iota_W( p_i\theta^i-\mathcal H ds)) = p_i\theta^i(W)|_{\Gamma(d)} - p_i\theta^i(W)|_\Gamma(c)
 \end{align*}
 where the first integral vanishes because $\Gamma$ is along a vortex line for a rod at equilibrium.
 
 But $\theta^i(W) = \pi^\star \beta^i(W) = \beta^i(\pi_\star W)$. Hence when the portion $(c,d)$ of the rod is displaced along the vector field $\pi_\star W$, a work is produced by the forces $p_i\,\beta^i$ applied at the tips of the portion and the potential energy varies exactly by that amount.

\subsection{Recipe for rod problems in framed and extended tangent bundles \label{sec-steps}}
Assuming we have a generalized rod --it could be a simple rod, a birod, an n-rod, or a rod constrained in some way -- such that the configuration of each section can be encoded as an element of a certain Lie group $Q$ of dimension $n\in \mathbb N$. The following steps lead one systematically from the energy functional to the equilibrium equations for the rod. As a by-product, the construction equips $\overline{TQ}$ with a Poisson structure; its consequences are 
not pursued here, but it underpins the study of conserved quantities and
stability which will be studied in sister papers. 

\begin{enumerate}
  \item
Compute a specific left-invariant frame of $Q$. One systematic way of doing so is to choose a faithful representation of $Q$. That is an injective homomorphism $\rho: Q\to \GL(k,\mathbb R)$ where $k\in\mathbb N$ is the dimension of the representation\footnote{The dimension of $Q$ and the dimension $k$ of the representation must be distinguished. For instance $SE(3)$ is a six-dimensional Lie group ($n=6$) with a standard faithful  representation in $\GL(4,\mathbb R)$ -- that is $k=4$}. We can then choose $n$ curves $\gamma_1,\cdots, \gamma_n$ defined on some interval $I\subset \mathbb R$ with $0$ in its interior and such that $\forall i: \gamma_i(0) = e$ and such that their tangent vectors be linearly independent. They are represented by $n$ curves $\mu_i(s) = \rho(\gamma_i(s))$ with linearly independent tangent vectors $b_i  := \tfrac{d}{ds} \mu_i(s)$. These vectors are therefore written as $n$ different $k\times k$ real matrices that are linearly independent and denoted $b_1,\cdots b_n$. 

The vector space generated by these matrices under standard matrix addition and multiplication by real numbers is by construction isomorphic to $T_e Q$. To get left-invariant vector fields, for any $M\in \mathfrak{Im}\rho$ we define the $n$ curves $M \, \mu_i(s)$ and their tangent vectors are then $M\, b_i$. The frame $B_1,\cdots, B_n$ of $Q$ that we used since section~\ref{sec-frame} can therefore be defined such that everywhere on $Q$, we have $\rho_\star(B_i|_q) = \rho(q)  b_i,$ where matrix multiplication is implied between $\rho(q)$ and $b_i$.

In particular if $V$ is a vector somewhere on $Q$, we can define its components as the numbers $\xi^i$ such that $\rho_\star (V|_q) = \xi^i \, \rho(q)\, b_i$. Because these numbers are uniquely defined for any vector $V$, they are indeed functions on $\overline{TQ}$: the quasi-velocities defined above.

\item Compute the structure constants $c_{ij}^k$. Note that this can be done systematically by expressing the matrix commutators of the $b_i$ as linear combinations of the basis : $[b_i,b_j]:= b_i\, b_j- b_j\,b_i = c^k_{ij} b_k.$
\end{enumerate}

The steps listed so far need only be done once for any Lie group pertinent to a specific rod problem. They do not depend on the specific constitutive laws of that rod. 

\begin{enumerate}
\setcounter{enumi}{2}
\item Write the energy functional in terms of the quasi-velocities $\xi^i$ and a suited expression of the configuration (any set of coordinates, or the chosen matrix representation, or another representation, or some other trick -- like quaternions to encode rotations).
\end{enumerate}

Next we take a moment to understand how to apply the operators $E_{q^i}$ and $E_{\xi^i}$ in practice. The $E_{q^i}$ derivatives are by definition vector fields that project to $B_i$ when acting on configurations but leave the $\xi^j$ unchanged since by duality we must have $d\xi^i(E_{q^j})=0$. For any function $\varphi:Q\to\mathbb R$, pulling back through the representation $\rho$ and extending to an open set in $GL(k)$ gives
\begin{equation}\label{derivativebyBi}
B_i \varphi|_q = \frac{\partial \varphi}{\partial M^c_d}\, \rho(q)^c_f\, (b_i)^f_d,
\end{equation}
where lower-case indices at the beginning of the alphabet run from 1 to $k$ and those in the middle run from 1 to $n$. So to compute $E_{q^j} \mathcal L$ we apply~\eqref{derivativebyBi} on any configuration and treat the quasi-velocities as constants for this differential operator. 

The $E_{\xi^i}$ operator project to $0$ under $\sigma_\star$ so they treat configuration information as constant and act like partial derivatives on expressions depending on the $\xi^i$. 

\begin{enumerate}
\setcounter{enumi}{3}
\item Define the conjugate momenta $p_i  =E_{\xi^i} \mathcal L$, the matrix-valued functions $P$ with image in $\mathbb R^{n\times n}$ and elements $P_{ij} = E_{\xi^i}E_{\xi^j} \mathcal L$, $Q$ with elements $Q_{ij} = E_{q^i}E_{\xi^j}\mathcal L$, and the line $\c t_j= E_s E_{\xi^j} \mathcal L$. Prepare the definition of the $D_{q^i}, D_p^i$ and $D_s$ according to section~\ref{sec-framechange}.

\item Compute the Hamiltonian function $\mathcal H = p_i \xi^i- \mathcal L$. If the change of variable $\xi^i\to p_i$ is practical then apply it here to get an explicit function of the configuration $q$ and the momenta. If the change is intractable or impractical, keep $\mathcal H$ as is and do not forget when computing Poisson brackets that $D_q$ and $D_p$ act on the functions $\xi^i$ according to the algebra described in section~\ref{sec-frame}.

\item Finally, write down the differential system that encodes the Hamiltonian flow. If the $p_i(\xi)$ functions can be inverted, solve the first two equations of~\eqref{hamflow}. If $Q$ is encoded by the representation $\rho$, then the $b_i$ are a set of $n$ particular constant matrices, and we directly have $\rho(q)' = \xi^i(p)\, \rho(q)\, b_i.$  
If however it is not possible or too cumbersome to invert $p_i(\xi)$, then we can solve for the whole differential system~\eqref{hamflow}.
\end{enumerate}

\section{One rod, three encodings\label{sec-singlerod}}

When the internal energy of the rod is the only strain-dependent term in the
potential energy, the components $p_i$ of the momentum covector $p_i\,\beta^i$
are the material force and moment $(\c n,\c m)$ introduced in
section~\ref{sec-cosserat}. This section shows how the machinery developed
above applies to a simple and well-studied problem: a weight hanging at the
tip of a rod. The very objects that are standard in rod theories built by
balancing forces and moments appear unprompted in the Hamiltonian picture,
provided that this picture is constructed as in section~\ref{sec-theory}. To keep
the tedium in check, we work under the unshearability and inextensibility
assumptions --- the Kirchhoff case, for which $Q=SO(3)$ and $p_i = \c m_i$;
the full Cosserat case proceeds identically on $SE(3)$ and is not carried
out here.

The configuration of every section is encoded by a rotation that applies the lab frame onto the material frame. So for this application, $Q$ is the rotation group. It has an obvious representation as $SO(3)$. So any rotation is encoded by the matrix $R\in SO(3)$ that can be applied to the column of the components of a three-dimensional vector in the lab frame to get the column of components of the rotated vector. 

Because of the unshearability and inextensibility assumptions, the centreline then obeys the differential equation $\v r' = \v d_3$ or in components $\v r' = R\,\v e_3$ where $\v e_3$ stands for the constant column $\v e_3  = \begin{pmatrix} 0 & 0 &1\end{pmatrix}^T$.

\textbf{Step 1.} Choose three curves passing through the identity 
\begin{align} 
  R_1(t) & = \begin{pmatrix} 1 & 0 & 0 \\ 0 & \cos t & -\sin t\\ 0 & \sin t & \cos t \end{pmatrix}, & 
R_2(t) & = \begin{pmatrix} \cos t & 0 & \sin t \\ 0 &1 & 0\\ - \sin t& 0  & \cos t \end{pmatrix}, & 
R_3(t) & = \begin{pmatrix} \cos t & -\sin t & 0 \\ \sin t & \cos t & 0 \\ 0 & 0 & 1 \end{pmatrix},
\end{align} 
and define the tangent vectors at the identity: 
\begin{align} \label{biso3}
  b_1 := R_1'(0) &=\begin{pmatrix} 0 & 0& 0 \\ 0 & 0& -1 \\ 0 & 1 & 0\end{pmatrix}& 
    b_2 := R_2'(0) &=\begin{pmatrix} 0 & 0& 1 \\ 0 & 0& 0 \\ -1 & 0 & 0\end{pmatrix}& 
      b_3 := R_3'(0) &=\begin{pmatrix} 0 & -1& 0 \\ 1 & 0& 0 \\ 0 & 0 & 0\end{pmatrix}& 
\end{align}
The chosen left-invariant vector fields on $Q$ are then defined by $B_i = R\, b_i$. Note that for this framed group, we find that $(b_i)_j^k$ is fully antisymmetric with $(b_1)^3_2 = 1$.

\textbf{Step 2.} Computing the commutators explicitly yield that $c^k_{ij}$ is the fully antisymmetric tensor such that $c^3_{12} = 1$.

For the rotations with this particular representation, the left-invariant vector fields correspond to rotation rates expressed in the local frame.
Therefore, the quasi velocities $\xi^i$ are the angular velocities, and we will denote $\xi^i = \c u^i$ to stick with the standard notation of rod theories.

As long as we have no reason to depart from the original choice of $b_i$, these first three steps need not be repeated for any rod problems on $Q=SO(3)$. The structure constant $c^k_{ij}$ that we have just established simply encode the cross product that is ubiquitous in computations with angular velocities. That is, for any two 3-tuples $a_k$ and $b^i$, we have that 
\begin{equation} \label{struct_cross}
  a_k b^i c^k_{ij} = (a\times b)_j.
\end{equation}

Given a different column $a^k$ containing the components of an angular velocity in the local frame, we can compute the operator that would take a derivative along a curve that admits that angular velocity. That is for any function $\varphi$ defined on $\overline{TQ}$, we have 
\begin{equation}\label{defineEa}
  E_a \varphi:= a^i\, E_{q^i} \varphi=a^i\, \pd{\varphi}{R^j_k} R^j_l (b_i)^l_k=\pd{\varphi}{R^j_k} R^j_l (a^\times)^l_k,
\end{equation}
where the last equality is due to the specific definitions of $b_i$ in~\eqref{biso3}.

\begin{rem}
Given the structure constants and regardless of other information in the Lagrangian, by substitution in~\eqref{hamflow} we already know that the Hamiltonian flow will take the form
\begin{equation} \label{hamflowso3}
  \begin{cases} 
    q' = B_i \,\c  u^i, \\ 
    p_i' = - D_{q^i} \mathcal H + (p\times \c u)_i.
  \end{cases} 
\end{equation}
\end{rem}

\textbf{Step 3.} The potential energy of the system is the sum of the elastic energy stored in the rod and the potential energy due to the height of the weight.
\[ E = \int_0^L  \underbrace{ \left\{\frac{K_{11}}{2} (\c u^1 - \widehat {\c u^1})^2 + \frac {K_{22}}2 (\c u^2 - \widehat {\c u^2})^2 +  \frac {K_{33}}2 (\c u^3 - \widehat {\c u^3})^2 + \overbrace{Mg\, \v d_3 \cdot \v e_3}^{\psi(q)} \right\}}_{\mathcal L} ~ds,\]
where $\v e_3$ is upward and unit vertical vector constant in the lab frame and $\v d_3$ is the third material director: the unit tangent to the centreline of the rod (see section~\eqref{sec-cosserat}). The last term depends on the orientation of the section but not on the strains. The precise expression of this term depends on the choice made to encode the rotation group; for now the notation $\psi(q)$ is there to indicate that it depends on the configuration but not on the quasi-velocities.

\textbf{Step 4.} Define 
\begin{equation}\label{defpso3}
p_i = E_{\xi^i} \mathcal L = K_{ij}\, (\c u^j - \widehat{\c u^j})\qquad  (\textrm{where }K\textrm{ is diagonal}).
\end{equation}

The $P$ matrix has components $P_{ij}  = K_{ij}$ and is therefore diagonal. Assuming the reference curvatures independent of arc length, we also have $Q = 0$ and $\c t = 0$. Accordingly, we have the definitions $D_{q^i} = E_{q^i}$ and $D_s = E_s$.

Since the definition~\eqref{defpso3} is precisely the strain-stress relations for unshearable and inextensible rods, we recognize that the generalized momenta are the components $\c m_i$ of $\v m$ in the material frame (see section~\ref{sec-cosserat}). To keep with standard notation, for the rest of this section we will substitute $\c m_i$ in place of $p_i$. 

In particular,  the second set of equations in~\eqref{hamflowso3} can be written as the advected derivative $\c m' + \c u\times \c m = - D_q \mathcal H.$

\textbf{Step 5.} Write the Hamiltonian function. Because~\eqref{defpso3} is easily inverted, we have 
\begin{equation} 
  \mathcal H = \c m_i \,  \widehat{u^i} + \frac {\c m_1^2}{2K_{11}} + \frac {\c m_2^2}{2K_{22}} + \frac {\c m_3^2}{2K_{33}} - \psi(q).
\end{equation}

\textbf{Step 6.} 
To write down an explicit differential system, we need to choose how we want to encode the rotation group.

\begin{enumerate}[(a)]
  \item $\mathbf{Q = SO(3).}$ Encoding with rotation matrices directly, we have  $\psi(R)  = Mg\, R_{3}^3$. So that applying~\eqref{defineEa}, and using $D_{q^i} = E_{q^i}$ since $Q =0$ and $\c t = 0$,
  \begin{align} \label{dqHso3}
    D_{q^i} \mathcal H&= -E_{q^i} \psi  = -M\,g \delta^3_j\, \delta^k_3 R^j_l (b_i)^l_k= M\,g  (\c e_3\times \c d_3)_i,
  \end{align}
  where $\c e_3$ stands for the components of the third lab frame vector expressed in the local frame and $\c d_3 = \begin{pmatrix} 0 & 0 & 1\end{pmatrix}^T$ is the component of the third material vector in the local frame.
  Substituting~\eqref{dqHso3} in~\eqref{hamflow} yields 
  \begin{equation}\label{sdso3}
    \begin{cases}
      R' = R\, \left (\c u \right )^\times, \\ 
      \c m' +\c u\times \c m + \c d_3 \times (-M g \c e_3) = 0,
    \end{cases}
  \end{equation}
  where $K$ is the bending stiffness matrix of the rod and where we recovered the Kirchhoff equations for the hanging weight (with $\v n = - M\,g\,\v e_3$ -- see section~\ref{sec-cosserat}).
  For a closed system, we can replace $\c u = \widehat{\c u}+K^{-1}\c m$ and the  $\c d_3 \times (-M g \c e_3) $ term of~\eqref{sdso3} by its penultimate expression in~\eqref{dqHso3}. 
  \item \textbf{Unit quaternions.} A detailed and wonderfully well written discussion of the link between rotations and quaternions can be found in~\cite{al86} where Rodrigues work gets its due. Here we will simply gather the key elements to enable our computation. When encoding the rotation group by quaternions, we associate a right-handed rotation of angle $\theta$ about a unit direction $\v a$ with the unit quaternions (i.e. quaternions of modulus 1)
   \begin{equation} \label{quaternionencoding}
    \pm \left(\cos\tfrac \theta 2, \sin \tfrac \theta 2 \v a\right).
  \end{equation} 
  In general, we note a quaternion $q = (q_0, \v q)$ as having a real part $q_0\in\mathbb R$ and a vector part $\v q= (q_1,\,q_2,\,q_3)\in\mathbb R^3$ (sometimes called the imaginary part of the quaternion). With those notations, if $a$ and $b$ are quaternions, the quaternion multiplication reads:
  \begin{align*}
    a\,b  = (a_0,\v a) (b_0,\v b) = (a_0\,b_0 -\v a\cdot \v b, a_0 \, \v b + b_0\, \v a + \v a\times \v b).
  \end{align*}
  
The quaternion algebra encodes rotation composition because it matches exactly the algebra of the Rodrigues parameters (see~\cite{al86}). It is also well known that along a curve in the unit quaternions $\mathbb H_1$ that is mapped to a rotation curve with angular velocity $\c u$ (that is $R' = R\, \c u^\times$), we have 
  \begin{equation}\label{quaternionderivative} 
    q' = \frac 1 2 q\, (0, \c u) ,
  \end{equation}
  where the factor $\tfrac 1 2 $ comes from the half angles in the encoding~\eqref{quaternionencoding}.

  In particular for a general function $\varphi$ defined on $\mathbb H$, we then have 
  \begin{equation}\label{buquat}
    B_u \varphi := u^i\, B_i \varphi  = u^i \pd{\varphi}{q_j}\, B_i \, q_j = \frac 1 2 \c u^i \pd{\varphi}{q_j}\, \Big(q\, (0,e_i)\Big)_j =\frac 1 2 \pd{\varphi}{q_j}\,\Big( q\, (0,\c u)\Big)_j. 
  \end{equation}
  Where $e_1,e_2,e_3$ form the canonical basis of $\mathbb R^3$.

  Finally let $q$ be a unit quaternion encoding a rotation with rotation matrix $R\in SO(3)$ and a vector of components $\v r$ in the lab frame, then we have  
  \begin{equation}\label{quaternionrot}
    (0, R\v r) = q\, (0,\v r)\, \overline q,
  \end{equation}
  where $\overline{(q_0,\v q)} = (q_0, -\v q)$ denotes the conjugate quaternion (which is also the multiplicative inverse in $\mathbb H_1$).

  With~\eqref{quaternionrot} and remembering that $\v d_3 = R\, \v e_3$, we find 
  \begin{equation}\label{quatpsi}
  \psi(q) = M\, g\, \v e_3 \cdot \v d_3 =M\, g\, (0, \v e_3)\cdot \Big( q\, (0, \v e_3)\, \overline q\Big) = M\, g\, (q_0^2 - q_1^2 -q_2^2 + q_3^2).
  \end{equation}

Applying~\eqref{quaternionderivative} yields $D_{q^i} q = \tfrac 1 2 q (0,e_i)$ and therefore
\begin{equation}\label{dqHquat}
D_{q^i} \mathcal H = - D_{q^i} \psi  = - \frac 1 2  Mg\, (0,e_3)\cdot \Big(q (0, e_i) (0,e_3)\overline q + q (0,e_3)\overline{q (0, e_i)}\Big) = - Mg\, (0, e_3)\cdot \Big(q \, (0,e_i\times e_3)\,\overline q\Big).
\end{equation}
 
Computing the quaternion multiplication, and gathering those components in a column, we find 
\begin{equation}\label{dqHquatcol}
  D_q \mathcal H =2\, Mg\, \begin{pmatrix} q_3 \, q_2 + q_0\, q_1 \\
  -q_1\, q_3 + q_0\, q_2\\
  0
  \end{pmatrix} .
\end{equation}

The full differential system~\eqref{hamflowso3} then reads 
\begin{equation}\label{sdquat}
\begin{cases} 
 & q_0'= -\frac 1 2 \v q \cdot \c u,\\ 
&\v q'  = \frac 1 2 (q_0 \c u+ \v q \times \c u),\\
&\c m' + \c u \times \c m = - D_q \mathcal H,
\end{cases}
\end{equation}
where $\c u$ and $D_q \mathcal H$ are placeholders respectively for $\widehat{\c u}+ K^{-1} \c m$ and~\eqref{dqHquatcol}.

\item \textbf{Cayley vectors.} Another way to encode rotations is to use Cayley vectors. The Cayley vector associated with a right-handed rotation of amplitude $\varphi \in (-\pi,\pi)$ and of axis along the unit vector $\v a$ is 
\[\eta  = \v a\, \tan (\varphi/2).\]
It is connected to the quaternions because the Cayley vector $\eta(q)$ of a rotation represented by the quaternion $q = (q_0,\v q)$ obeys $\eta(q) = \v q/q_0$: that is, it is the projection of the unit quaternion from the origin of the full quaternion space to the hyperplane of equation $q_0 =1$. That relation can also be inverted: a rotation of Cayley vector $\eta$ admits the unit quaternion
\begin{equation} \label{cayleytoquat}
  q(\eta)= \left (\frac{1}{\sqrt{1+ |\eta|^2}}, \frac {\eta} {\sqrt{1+|\eta|^2}}\right )
\end{equation}

Cayley vectors have the nice property that they form a chart of coordinates that almost covers the whole of $SO(3)$; only those rotations of amplitude $\pi+2 k\,\pi$ with $k\in\mathbb Z$ must be excluded. Furthermore, they are coordinates that behave well under composition. If $\eta_1$ and $\eta_2$ are Cayley vectors for two rotations, the Cayley vector for the composition of these rotations is 
\begin{equation} \label{cayleycomp}
  \frac{\eta_1 + \eta_2 + \eta_1 \times \eta_2}{1-\eta_1\cdot \eta_2},
\end{equation}
which only contain rational operations. The associated rotation matrix can be recovered via the Cayley transform: 
\begin{equation} \label{cayleytoso3}
  R(\eta) = (I+\eta^\times)(I-\eta^\times)^{-1}.
\end{equation}

A Cayley vector $\eta$ rotates a vector of components $\v b$ in lab frame to a vector of components 
\begin{equation}\label{cayleyact}
  R(\eta)\v b = \frac{(1-|\eta|^2)\,{\v b} + 2(\eta\cdot {\v b})\,\eta + 2\,\eta\times {\v b}}{1+|\eta|^2}
\end{equation}
in the lab frame.

Finally along a curve of quasi-velocity $\c u$, we have 
\begin{equation} \label{cayleyder}
   \eta' =\frac12\Big({\c u} + \eta\times {\c u} + (\eta\cdot {\c u})\,\eta\Big).
\end{equation}

Therefore, using~\eqref{cayleyact} to express $\v d_3 = R\,\v e_3$ in the definition of $\psi$, we get 
\begin{equation} 
\psi(q) = Mg\, \frac {1 -\eta_1^2 - \eta_2^2 + \eta_3^2}{1+\eta_1^2 + \eta_2^2 + \eta_3^2}.
\end{equation}
Using this last expression to infer $D_q\mathcal H$ would in truth be tedious. But we can use the fact that we have already done the job in~\eqref{dqHquatcol} and we can translate it using~\eqref{cayleytoquat}: 
\begin{equation} 
  D_q\mathcal H  = \frac{2\,Mg}{1+ |\eta|^2} \begin{pmatrix} 
\eta_3\, \eta_2 + \eta_1\\
-\eta_1\,\eta_3 +\eta_2\\
0
  \end{pmatrix}.
\end{equation}

The complete differential system~\eqref{hamflowso3} then reads 
\begin{equation} \label{sdcayley}
  \begin{cases}
  \eta' = \frac12\Big({\c u} + \eta\times {\c u} + (\eta\cdot {\c u})\,\eta\Big),\\
  \c m' + \c u\times \c m = - D_q \mathcal H,
  \end{cases}
\end{equation}
where once again $\c u$ and $D_q\mathcal H$ are placeholders.

\begin{rem}
Note the path not taken. When working with coordinates, standard Hamiltonian calculus requires expressing $\mathcal L(\eta, \eta')$ instead of $\mathcal L(\eta, \c u)$ and defining the conjugated momenta $p_i = \pd{\mathcal L}{\eta_i'}$. Those who have tried know that the ensuing momenta are unrecognizable and that is to say nothing of what would have happened had we chosen Euler-angles rather than Cayley vectors. Instead, we engage on the path trodden by Poincaré and Hamel in the Lagrangian picture, and we explore its Hamiltonian staging post: the Hamiltonian function, the associated Poisson bracket and its bivector $\mathcal J$ are expressed within a framed space adapted to the structure of the Lie group. The
rewards of the detour are that the momenta we work with are physically explicit, and that the frame changes are linear even when the Legendre transform is not. 
\end{rem}
\end{enumerate}

\begin{rem}
The three encodings above are not three theories but three different strategies for describing the one space $\overline{TQ}$. The objects that carry the physics---the form
$\Theta$, the bivector $\mathcal J$, and the Hamiltonian $\mathcal H$---are defined on
$\overline{TQ}$ itself and never consult the encoding: rotation matrices, unit
quaternions, and Cayley vectors are interchangeable, and the construction
is indifferent to the choice. This is why the same momenta $\c m$ and the
same equilibrium equations emerged in all three. The encoding is a convenience for
writing the functional down; it is not part of the structure.
\end{rem}

Of course, none of the equations~(\ref{sdso3},\ref{sdquat},\ref{sdcayley}) are new. The point is precisely to show that the Hamiltonian recipe of section~\ref{sec-steps} gets us from the potential energy to the Kirchhoff equations with indecent ease.

\section{Two momenta, one bivector\label{sec-interact}}

\subsection{Decomposed momenta, mixed bivector frames\label{sec-splittheory}}

When dealing with an interaction that depends on the strains, the energy and the momenta are decomposed along subsystem lines; the bivector is not decomposed but expressed asymmetrically, its two factors projected in different frames.

Throughout this section we work in the rotation group $Q=SO(3)$, framed as in
section~\ref{sec-singlerod}: the structure constants are those of the
cross-product~\eqref{struct_cross}, the quasi-velocities $\xi^i=\c u^i$ are the
body-frame angular velocities, and the Kirchhoff momenta are
$\c m_i = K_{ij}(\c u^j - \widehat{\c u}^j)$, with $K$ the (invertible) bending
stiffness matrix.

Because energy is an additive quantity, there is a class of rod problems for which the energy density of the structure can be separated as 
\begin{equation}\label{interactinglag}
  \mathcal L = \overbrace{\frac 1 2 \left (\c u-\widehat{\c u}\right ) \cdot K\left (\c u-\widehat{\c u}\right )}^W + \mathcal L_1(q,\,\c u),
\end{equation}
 where $W$ is the elastic energy density in a standard Kirchhoff rod and where $\mathcal L_1(q,\c u)$ is an interaction term that also depends on $\c u$.

As an example, we will consider in section~\ref{sec-tendon} an elastic rod equipped with a tendon, that is an inextensible string that runs down the length of the rod and is set at a fixed offset $a$ from the centreline. An operator can apply a tension $T$ to the string resulting in the rod bending and twisting to accommodate some retraction of the tendon. 

Here we establish the general result. But before we dive in, note that because $\mathcal L_1$ depends on the strains, this situation destroys our identification of $p_i := E_{\c u^i} \mathcal L$ with the moment $\c m$ of Kirchhoff theories which would be $\c m_i = E_{\c u^i} W$. Worse, in many cases and it will be true for our tendon example, the relations between $p_i$ and the quasi-velocities $\c u^i$ become nonlinear\footnote{We have seen above one way of handling this, but here we can do better.}.

The functions $\c m_i$ on $\overline{TQ}$ are defined by $\c m_i = E_{{\c u}^i}W$. Because stiffness matrices are typically invertible, we do have that $(\theta^1,\cdots , \theta^n,d{\c m}_1,\cdots, d{\c m}_n, ds )$ forms a coframe on $\overline{TQ}$ and that the functions $\c m_i$ provide coordinates on every fibre. That coframe has its own dual frame $(C_{q^1},\cdots, C_{q^n},C_{\c m}^1,\cdots, C_{\c m}^n, C_s)$. The algebra allowing to change basis from the $E$ frame defined in section~\ref{sec-frame} and the $C$ frame that was just introduced is similar to that of section~\ref{sec-framechange}. 

The idea is not however to make a full change of frame where we would reexpress the whole theory in the $C$ frame. Indeed, we have shown that the Hamiltonian structure does not depend on our choices of coordinates, the functions that appear in proposition~\ref{prop-pforced} are the $p_i$ and not the $\c m_i$. But the Poisson bracket is defined by a bivector $\mathcal J$ that we wrote as a sum of tensor products. The idea is therefore rather to express the left factors of $\mathcal J$ in the $C$ frame (a linear change of frame) and to keep the right factors expressed in the $D$ frame. This is because in the Poisson bracket, the right factors act on the Hamiltonian function and we already know that we always have $D_p^i\mathcal H = \c u^i$ and it would be unfortunate to obfuscate such a simple relation behind an ill-adapted frame.

Beyond the new way of expressing $\mathcal J$, the quasi-velocities $\c u^i$ and indeed the momenta $p_i$ become explicit functions of the $\c m_j$. To this end, we also define $l_i = E_{{\c u}^i} \mathcal L_1$ so that by definition $p_i = E_{{\c u}^i} \mathcal L = \c m_i+ l_i$. In general we have nonlinear expressions of the form $l_i(q,\c u)$ but since the quasi-velocities $\c u$ are affine functions of the internal stresses $\c m$, we can express $p_i =\c m_i + l_i(q,\c m_i).$ Our Poisson bracket is then ready to produce the evolution equation of $q$ and $\c m_i$, or for that matter of any other quantities of interest, along the length of a rod at equilibrium.

\begin{prop}
Consider a rod problem with configuration $R\in SO(3)$ and a potential energy density of the form
\begin{equation}
\mathcal L(R,\c u,s) =  \overbrace{\frac 1 2 \left (\c u-\widehat{\c u}(s)\right ) \cdot K\left (\c u-\widehat{\c u}(s)\right )}^{W(\c u)} +\mathcal L_1 (R,\c u, s)
\end{equation} and such that $P_{ij} = E_{\xi^i}E_{\xi^j} \mathcal L$ and $K$ are  invertible matrices. Define also
\begin{align} 
   \c m_i &= E_{\xi^i} W,&
    l_i &= E_{\xi^i} \mathcal L_1,&
    L_{ij} &= E_{\xi^i}E_{\xi^j} \mathcal L_1,&
    Q_{ij} &= E_{q^i}E_{\xi^j}\mathcal L_1,&
    \widetilde{\c t} &= - K{\widehat{\c u}}'&
    \c t_i &= E_s E_{\xi^j} \mathcal L_1 + \widetilde{\c t}_i. 
\end{align}
Then $p_i = \c m_i + l_i$ and the matrix-valued function ${\mathcal G}(q,m)=(I+ K^{-1} L)^{-1}$ on $\overline{TQ}$ is such that the Hamiltonian flow of the problem is fully defined by 
\begin{equation} \label{hamflowinteract}
  \begin{cases} 
    R' &= R \,(\widehat{\c u} + K^{-1} \c m)^\times,\\ 
    \c m_i' &=  \widetilde{\c t}_i - \Big( (\c u\times p)_j -E_{q^j} \mathcal L_1 - \c m_k (Q{\mathcal G}K^{-1})^k_j +\c u^k Q_{kj} + \c t_j \Big ){\mathcal G}_i^j.
  \end{cases}
\end{equation} 
\end{prop}
\begin{proof} 
  First note that by definition $\c m_i = K_{ij} (\c u^j - {\widehat{\c u}}^j)$ which can be inverted to 
  \begin{equation}\label{uofm}
    \c u = \widehat{\c u} +K^{-1} \c m.
  \end{equation} So any function of $(R,\c u)$ can be written as a function of $(R,\c m)$ by directly substituting $\c u$ according to~\eqref{uofm}.

  We have already seen that in $SO(3)$, applying~\eqref{evolve} to $R$ yields 
  \begin{equation} 
    R' = \{R,\mathcal H\} + D_s R = \{R,\mathcal H\} = R (\c u)^\times \stackrel{\eqref{uofm}}= R\,(\widehat{\c u} + K^{-1} \c m)^\times,
  \end{equation}
  which is the first equation in~\eqref{hamflowinteract}.

  Next we establish the relation between the $C$ frame defined before the proposition as the dual frame to $(\theta^1,\cdots,\theta^n,\c m_1,\,\cdots,\, \c m_n, ds)$ and the $E$ frame defined in section~\ref{sec-frame} by following the same steps as in section~\ref{sec-framechange}. We define the matrices $\widetilde P_{ij} = E_{\xi^i}E_{\xi^j} W = K_{ij}$ and $\widetilde Q_{ij} = E_{q^i} E_{\xi^j}  W = 0$ and the line $\widetilde {\c t}_i = E_s E_{\xi^i} W = - K_{ij} (\widehat{\c u}^j)'$. Then applying~\eqref{computeDframe} for the case of our $C$ frame, we get
\begin{align}
  \begin{pmatrix} 
    E_{q}\\
    E_\xi\\
    E_s 
  \end{pmatrix} = \begin{pmatrix} I &0 &0 \\ 0& K &0 \\ 0 &\widetilde {\c t} & 1 \end{pmatrix} 
  \begin{pmatrix}
    C_q\\
    C_m\\
    C_s 
  \end{pmatrix} \qquad \textrm{so that}\qquad 
  \begin{pmatrix} 
    D_q \\ D_p \\ D_s \end{pmatrix}
    & \stackrel{\eqref{computeDframe}}{=} \begin{pmatrix} 
      I & - QP^{-1} & 0 \\ 
      0 &  P^{-1} & 0 \\
      0 & -\c t P^{-1} & 1\end{pmatrix} 
       \begin{pmatrix} I &0 &0 \\ 0& K &0 \\ 0 &\widetilde {\c t} & 1 \end{pmatrix} 
  \begin{pmatrix}
    C_q\\
    C_m\\
    C_s 
  \end{pmatrix}\\
  &= 
   \begin{pmatrix} I &-Q{\mathcal G} &0 \\ 0& {\mathcal G} &0 \\ 0 &\widetilde{\c t}-\c t {\mathcal G} & 1 \end{pmatrix} 
  \begin{pmatrix}
    C_q\\
    C_m\\
    C_s 
  \end{pmatrix}, \label{dtocc}
\end{align}
where we used 
\begin{align}\label{Gintroduced}
 P^{-1} K = (K+L)^{-1} K = \big(K (I+ K^{-1}L)\big)^{-1} K = (I+ K^{-1}L)^{-1}=:{\mathcal G}.
\end{align}

 The bivector $\mathcal J$ defined in~\eqref{defJ} and reproduced here for convenience can therefore be expressed as 
  \begin{align}
  \mathcal J& = D_{q^i}\otimes D_{p}^i - D_p^i\otimes D_{q^i}  + p_a \, c^a_{ij} D_p^j\otimes D_p^i,\\
& = C_{q^i} \otimes D_p^i - {\mathcal G}^i_k \,C_m^k\otimes D_{q^i}  + (p_a \, c^a_{ij}-Q_{ij}) {\mathcal G}^j_k\,  C_m^k\otimes D_p^i, \label{splitJ}
  \end{align}
  where we ought to remember that $Cs$ and $Ds$ are vector fields while $Q$ and ${\mathcal G}$ are matrix-valued functions. As promised the left factors are now written in the $C$ basis (because they are intended to act on functions of $(R,\c m)$) and the right factors are written in the $D$ basis because they are intended to act on the function $\mathcal H(q(R),p)$.

  As a last preparation, the Hamiltonian function is 
  \begin{equation}
    \mathcal H = p_i \, \xi^i - W- \mathcal L_1, 
  \end{equation}
  so that 
  \begin{align}
  D_{q^j} \mathcal H &= p_i \, D_{q^j} \xi^i - D_{q^j} \mathcal L_1,\nonumber\\
  &=-p_i  (QP^{-1})^i_j -E_{q^j} \mathcal L_1 + (QP^{-1})^k_j l_k,\nonumber\\
&=- \left(C_{q^j}\mathcal L_1 + \c m_i (Q{\mathcal G}K^{-1})^i_j\right).\label{DqHinteract}
  \end{align}

  We are now ready to compute 
  \begin{align}
  \c m_i' &\stackrel{\eqref{evolve}}= \{ \c m_i,\mathcal H\} + D_s \c m_i  \stackrel{(\ref{defbra},\ref{dtocc})}=\mathcal J(d\c m_i,d\mathcal H)+ \left (\widetilde {\c t}-\c t {\mathcal G}\right )_i,\nonumber\\
  &\stackrel{(\ref{splitJ},\ref{rDpH})}=- {\mathcal G}_i^j (D_{q^j}\mathcal H) + \left (p_a c^a_{kj}-Q_{kj}\right ) {\mathcal G}^j_i \c u^k +\widetilde{\c t}_i -\c t_j {\mathcal G}^j_i,\nonumber\\
  &= \widetilde t_i - \Big(D_{q^j}\mathcal H+ (\c u\times p)_j+ \c u^k Q_{kj} +\c t_j \Big)\, {\mathcal G}^j_i,\nonumber\\
  &\stackrel{\eqref{DqHinteract}}= \widetilde{\c t}_i - \Big( (\c u\times p)_j -E_{q^j} \mathcal L_1 - \c m_k (Q{\mathcal G}K^{-1})^k_j +\c u^k Q_{kj} + \c t_j ){\mathcal G}_i^j.
  \end{align}

\end{proof}

\begin{rem}
The ensuing system may seem involved, but we are paying for generality. In many practical cases we will have $\widetilde{\c t} =\c t = 0$ and/or $Q=0$ in which case the system considerably simplifies. 
\end{rem}
\begin{rem}
   The expected $\c m' + \c u\times \c m$ does not appear explicitly. But it is in fact hidden in $\c u\times p = \c u \times (\c m + l)$. Yes, it's multiplied by $\mathcal G$, but we can always set $\mathcal G = I+ A$ and the $\c u\times \c m$ term will become explicit at the cost of a more cumbersome general expression. 
\end{rem} 
\begin{rem}
   The matrix $\mathcal G$ can be computed fully by rational operations only. By the Cayley-Hamilton theorem trick, for any matrix $M$ such that $\det(I+M)\neq0$,
  $$(I+ M)^{-1} = \frac {M^2 -(I_1+1) M + (I_1+ I_2) \textrm{Id}}{1+I_1+ I_2-I_3},$$
  where $I_1 = \tr M$, $I_2 = ((\tr M)^2-\tr (M^2))/2$ and $I_3  = \det M$ are the matrix invariants of $M$. In the example to follow, we will see that for that problem, much can be gained by studying the matrix $(K^{-1}L)$ in some detail.
\end{rem}

\subsection{The application\label{sec-tendon}}

\begin{figure}[t]
\begin{center}
\includegraphics[width=0.49\textwidth]{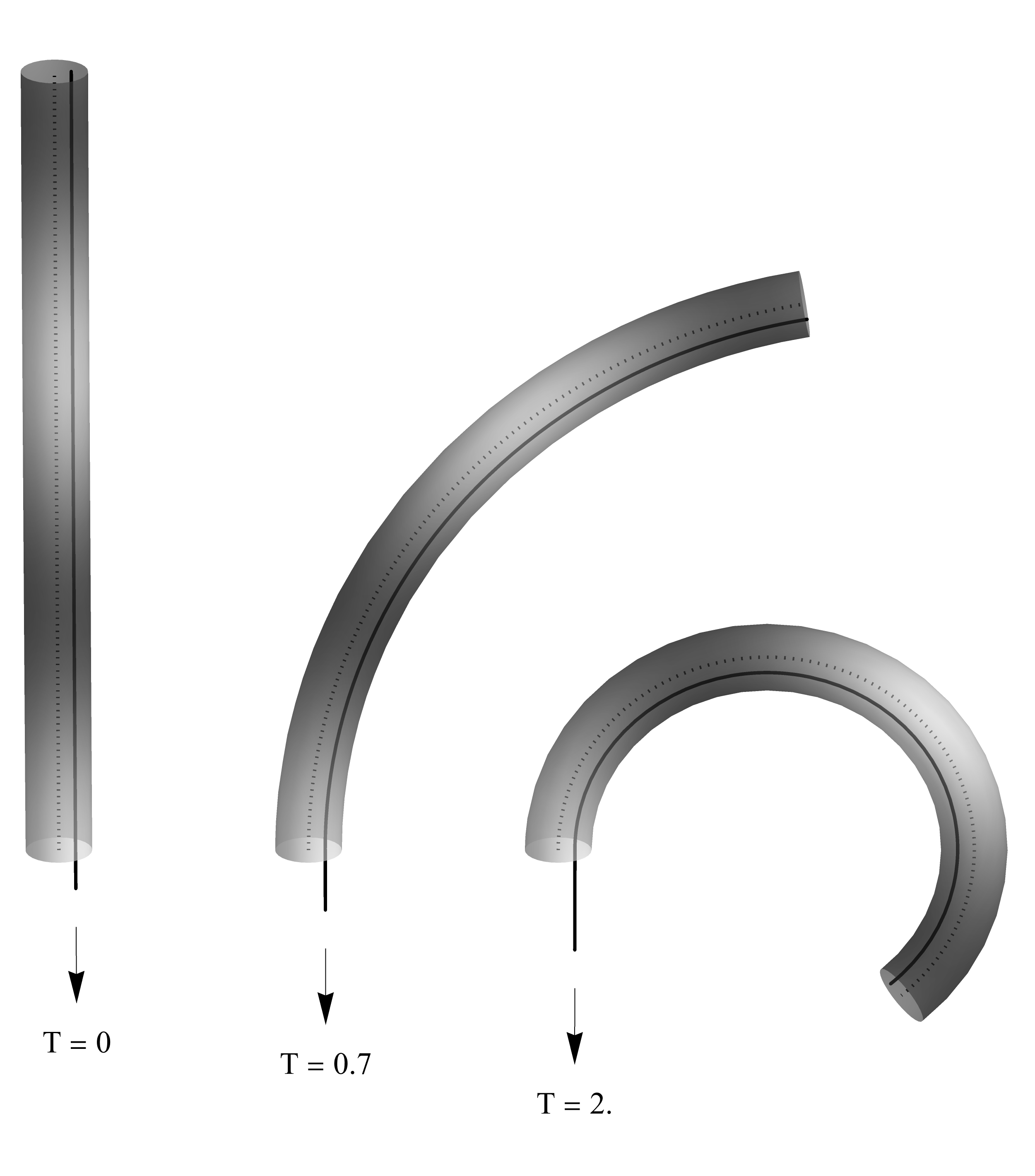}\hfill
\includegraphics[width=0.49\textwidth]{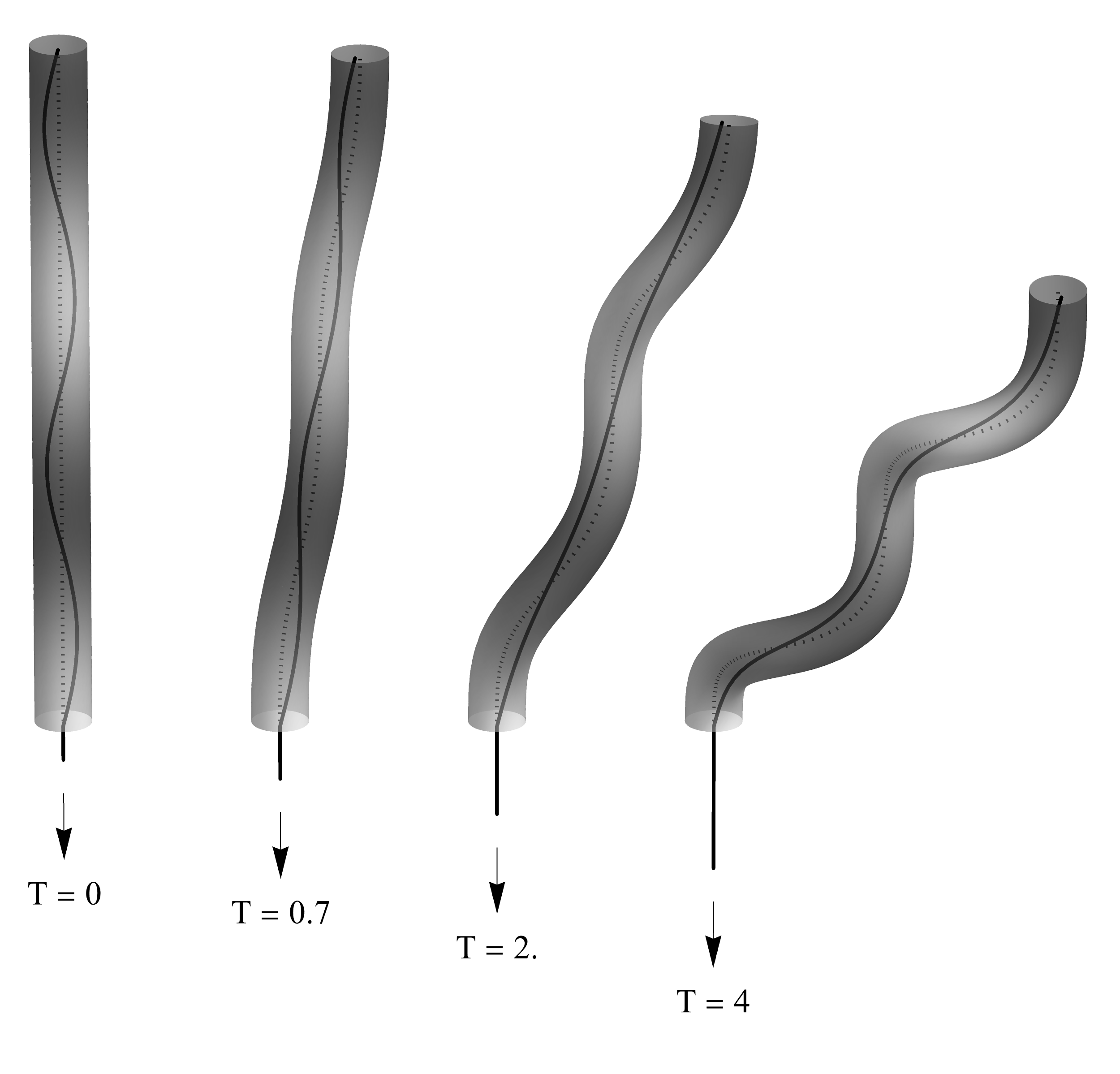}
\caption{
Equilibrium configurations of a straight rod actuated by a tendon ($K_{11}=K_{22}=1,\, K_{33}=0.8,\, a=0.8,\, L = 10.$). The centreline of the flexible rod is shown as a dotted line, the tendon is the thicker black line. (left) Straight reference configuration ($\widehat{\c u}^1=\widehat{\c u}^2=\widehat{\c u}^3=0$). (right) Twisted reference configuration ($\widehat{\c u}^1=\widehat{\c u}^2=0,\, \widehat{\c u}^3=4\pi/L$). Note that the tendon within the rod straightens up first (around $T=2$) but it can be further withdrawn by bending the structure for higher values of the tension $T$.}\label{fig-straight}
\end{center}
\end{figure}

An unshearable and inextensible Kirchhoff rod with a diagonal bending stiffness matrix $K$ is equipped with a tendon. That tendon is an inextensible string --a rod with zero bending stiffness-- inserted in a tunnel within the rod at a distance $a$ from the centreline of the rod along the $\v d_1$ director. It enters the rod at $s=0$ runs along the sections $[0,L)$ where it is free to slide and has its tip fastened to the section at $s = L$. By pulling on it at the entrance $s=0$, an operator can apply a tension $T$ to the tendon. 

Our aim is to display the program of section~\ref{sec-splittheory} at work on a nontrivial example. It is not to provide a systematic study of tendon actuated robots. In particular, we assume that there is no friction between the tendon and the elastic rod. Detailed studies of this system are available for instance in~\cite{bole23} from the viewpoint of optimal control. They work with the left-invariant frame on the Lagrangian side but do not develop the Hamiltonian picture. A different approach, based on balancing the inner forces within the system is explored in~\cite{ruwe14}.

The position $\v r_t(s)$ of the tendon passing section $s$ is
\begin{equation} \label{defrt}
  \v r_t  = \v r + a\, \v d_1,
\end{equation} 
and the element of arc length along the tendon line is 
\begin{equation} \label{dtendon}
\Vert \v r_t'\Vert ds = \Vert \v d_3 + a\, \v u \times \v d_1 \Vert ds = \Vert (1-a\,\c u^2) \v d_3 + a\,\c u^3\v d_2\Vert ds = \sqrt{(1-a\c u^2)^2 + (a\,\c u^3)^2} ds,
\end{equation}
where $'$ denotes derivative w.r.t. $s$ the arc length along the centreline of the Kirchhoff rod.

The potential energy of the system can therefore be written as 
\begin{align} 
  E &= \int_0^L \Big\{\frac 1 2 (\c u-\widehat{\c u}) \cdot K (\c u-\widehat{\c u}) + T\,\Vert \v r_t'\Vert \Big\} ds,\nonumber \\
&=\int_0^L \underbrace{\Big\{\underbrace{\frac 1 2 (\c u-\widehat{\c u}) \cdot K (\c u-\widehat{\c u})}_W + \underbrace{T\,\sqrt{(1-a\c u^2)^2 + (a\,\c u^3)^2}}_{\mathcal L_1} \Big\}}_{\mathcal L} ds,\label{entendon}
\end{align}
where in general we allow for $\widehat {\c u}$ to depend explicitly on $s$ encoding a general reference shape of the Kirchhoff rod.

Because, we can cast the potential energy of the system as a function on the extended tangent bundle of the rotation group, \textbf{Step 1.} and \textbf{Step 2.} of the recipe in section~\ref{sec-steps} were already taken care of in section~\ref{sec-singlerod}. We keep the same framing and so we have the same structure constants.

By definition, $p_i = E_{\xi^i} \mathcal L$ but as per the plan established in section~\ref{sec-splittheory}, we also define 
\begin{align}\label{defmtendon}
  \c m_i &= E_{\xi^i} W = K_{ij}(\c u^j -\widehat {\c u}^j),&
  l_i & = E_{\xi^i} \mathcal L_1. 
\end{align} 
Then we establish the relation between the $C$ framed defined in section~\ref{sec-splittheory} and the $E$ frame defined in section~\ref{sec-frame}. Here, we have
\begin{align} \label{Qttendon}
  Q &= 0,& \widetilde {\c t}_i &= \c t_i = -K_{ij} {\widehat{\c u}^{j\prime}}.
\end{align}

Next, we need to compute the $\mathcal G = (I+ K^{-1}L)^{-1}$ matrix. To express $L_{ij}=E_{\xi^i}E_{\xi^j}\mathcal L_1=E_{\xi^j} \, l_i$ explicitly, we prepare 
\begin{align} \label{defLambda}
  \mathcal L_1 &= T \, \underbrace{\sqrt{(1-a\,\c u^2)^2 + (a\,\c u^3)^2}}_{\Lambda(\c u)},& 
  \lambda_1 & = 0,&
  \lambda_2&= 1-a\,\c u^2,&
  \lambda_3&=a\, \c u^3.
\end{align}
Then we can compute 
\begin{align}\label{defltendon}
l_1 &= E_{{\xi^1}}\mathcal L_1 = 0,&
l_2 &= E_{\xi^2}\mathcal L_1 =T\pd{\Lambda}{\c u^2}= -a\, T\frac{\lambda_2}{\Lambda},&
l_3 &= E_{\xi^3}\mathcal L_1=T\pd{\Lambda}{\c u^3}= a\, T\frac{\lambda_3}{\Lambda}.
\end{align}
Before we compute the second derivatives, remark that the $l_i= E_{\xi^i} \mathcal L_1$ have a physical meaning too. They are the moment of force w.r.t. the rod's centreline $\v r$ generated by the tension in the tendon. To see it, remark that the unit tangent to the tendon $\v r_t$ is $\v t = \v r_t'/\Vert \v r_t'\Vert  =( \v d_3+ a\v u\times \v d_1)/\Lambda$. So the force transmitted by the tendon across a section $s$ is therefore $ T\v t$ and it generates a moment $\boldsymbol \mu$ with respect to $\v r$ given by 
\begin{align*}\boldsymbol \mu = a\v d_1 \times (T\v t) = \frac {a \, T}{\Lambda}\, \v d_1 \times (\v d_3 +a \v u\times \v d_1) = l_1\, \v d_1+ l_2\,\v d_2 + l_3 \, \v d_3,
\end{align*}
so that $l$ is the column of the components of $\boldsymbol \mu$ in the material frame.

For the second derivatives, because the $E_{\xi^i}$ vector fields commute we need only compute 
\begin{align}\label{Lcomps}
L_{22} &= \pd{l_2}{\c u^2} = \frac{a^2\, T}{\Lambda} \left ( 1- \frac {\lambda_2^2}{\Lambda^2}\right),&
L_{32} &= \pd{l_2}{\c u^3} = \frac{a^2\, T}{\Lambda} \,\frac{\lambda_2\, \lambda_3}{\Lambda^2},&
L_{33} &= \pd{l_3}{\c u^3} = \frac{a^2\, T}{\Lambda}\, \left ( 1- \frac {\lambda_3^2}{\Lambda^2}\right),\\
&=\frac {a^2\, T}{\Lambda^3} \lambda_3^2=\frac{l_3^2}{T\,\Lambda}, & 
&=-\frac{l_2\,l_3}{T\Lambda} & 
&=\frac {a^2\, T}{\Lambda^3} \lambda_2^2 =\frac{l_2^2}{T\Lambda},
\end{align}
where the second line comes because $\Lambda^2 = \lambda_2^2+ \lambda_3^2$. 

This can be written in a tidy way by gathering  $l = \begin{pmatrix} l_1 & l_2 & l_3\end{pmatrix}^T$ and noticing that 
\begin{equation}\label{Lexpressed} 
  L = \frac 1 {\mathcal L_1} (\c d_1\times l) \otimes (\c d_1 \times l)\qquad \textrm{and}\qquad
   K^{-1}L = \frac 1 {\mathcal L_1}\, \big(K^{-1}(\c d_1\times l) \big ) \otimes (\c d_1\times l) = \frac 1 {K_{22}\, K_{33}\mathcal L_1}\, (\c d_1\times Kl) \otimes (\c d_1\times l),
\end{equation}
where in keeping with our sans-serif convention, we have $\c d_1:=\begin{pmatrix} 1 & 0 &0 \end{pmatrix}^T$. 

Also, by summing diagonal components in~\eqref{Lcomps} we compute
\begin{align}\label{thetrace}
  \tr K^{-1}L = \frac {a^2 T} {K_{22} K_{33}\,\left (\lambda\cdot \lambda\right)^{3/2}} \lambda\cdot K\lambda = \frac{1}{K_{22}\,K_{33}\,\mathcal L_1} l\cdot Kl.
\end{align}
where $\lambda = \begin{pmatrix} \lambda_1 & \lambda_2 & \lambda_3\end{pmatrix}^T$.

Notice that $L$ has a block diagonal structure with a $1\times1$ block that is 0 and $2\times 2$ block that has zero determinant. Accordingly, with $M_2$ the non-trivial $2\times 2$ block of $K^{-1} L$, we have $\det M_2 = 0$ and
\begin{align}\label{Gtendonblock}
  \mathcal G = \left (I_{3\times3}+ K^{-1}L\right )^{-1} = 
  \begin{pmatrix} 1 & 0 \\
    0& I_{2\times 2} + M_2\end{pmatrix}^{-1} = 
    \begin{pmatrix} 1 & 0 \\ 0 & (I_{2\times2} + M_2)^{-1} \end{pmatrix}
\end{align} 
The $2\times 2$ matrix can be inverted by Cayley-Hamilton theorem trickery: 
\begin{align*} 
(  I_{2\times 2} + M_2)^{-1} = I_{2\times2} - \frac 1 {1+\tr M_2} M_2,
\end{align*} 
where we used $\det M_2 =0.$
In conclusion, we have 
\begin{align} \label{Gtendon}
  \mathcal G = I_{3\times3} - \frac 1 {1+ \tr K^{-1}L} \, K^{-1} L=I-\frac 1 {K_{22}K_{33} \mathcal L_1+ l\cdot K l}\, (\c d_1 \times Kl)\otimes(\c d_1 \times l).
\end{align}

\begin{figure}[t]
\begin{center}
\includegraphics[width=0.49\textwidth]{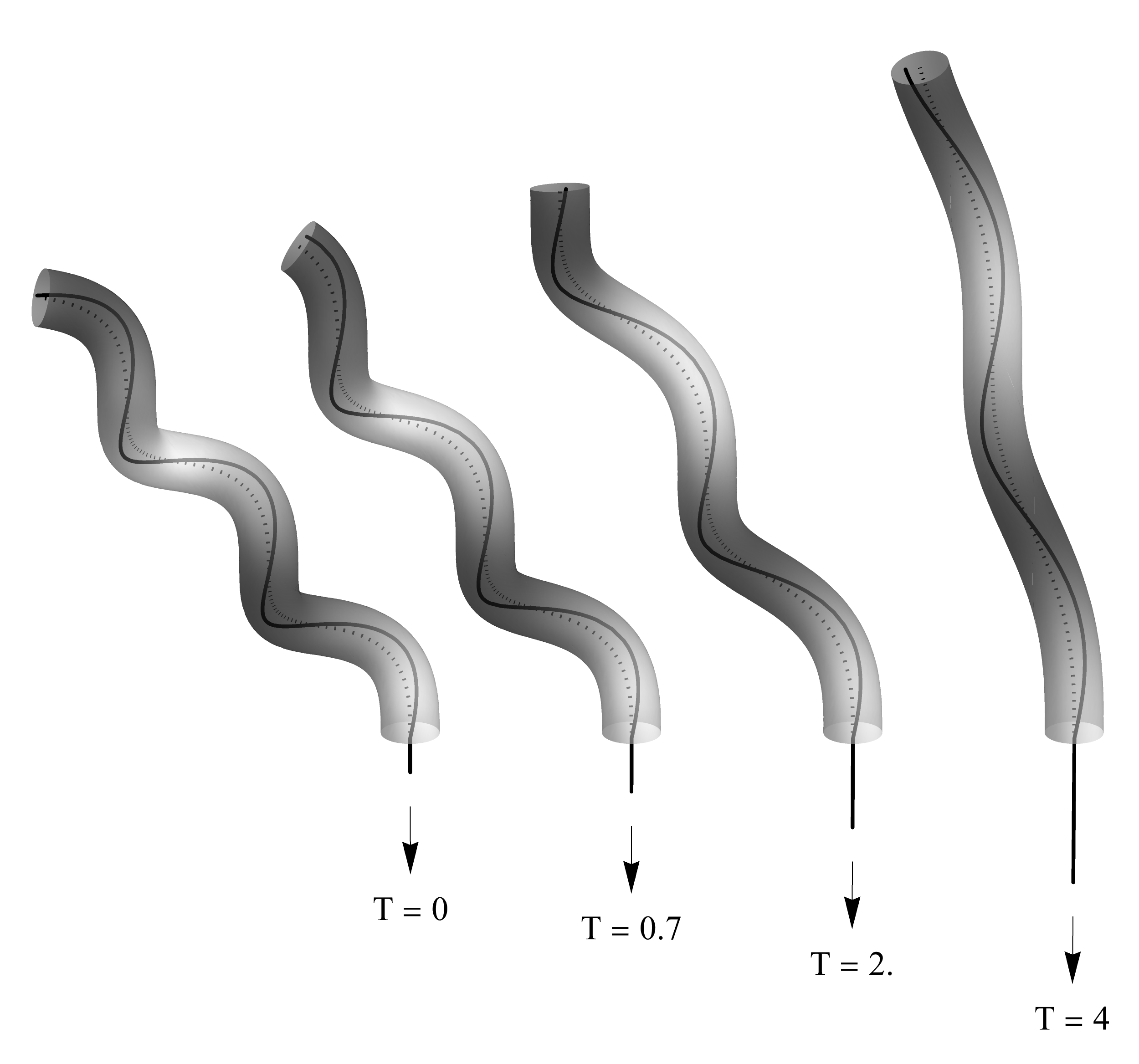}\hfill
  \includegraphics[width=0.49\textwidth]{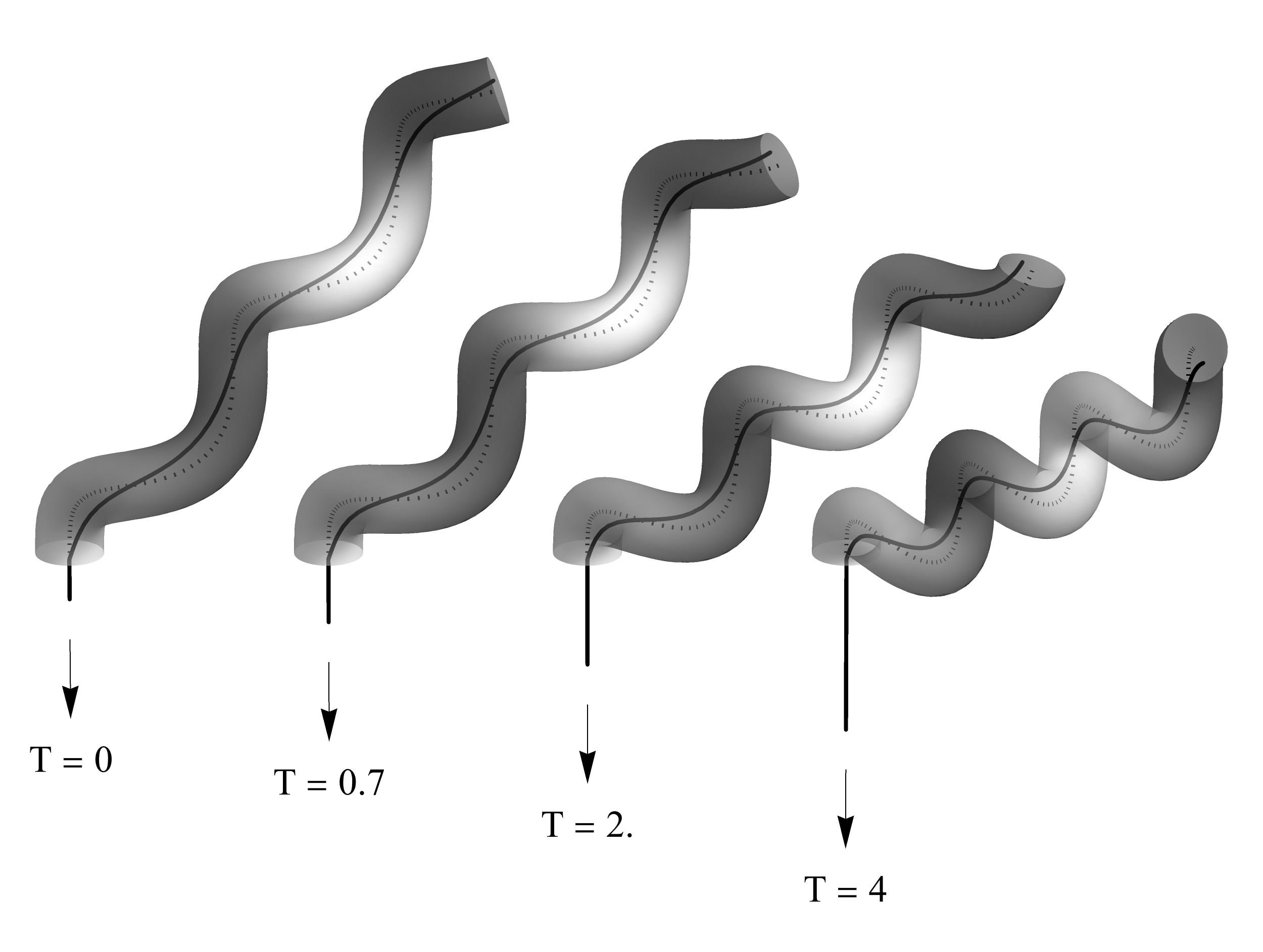}
\caption{Equilibrium configurations of a helical rod actuated by a tendon ($K_{11}=K_{22}=1,\, K_{33}=0.8,\, a=0.8,\, L = 10.$). (left) The reference configuration is a left helix with constant $\widehat{\c u}^1=0, \widehat{\c u}^2=-1,$ and $\widehat{\c u}^3=4\pi/L$: the tendon is in the extrados. (right) The reference configuration is a right helix with constant $\widehat{\c u}^1=0, \widehat{\c u}^2=1,$ and $\widehat{\c u}^3=4\pi/L$: the tendon is on the intrados. With a tendon in the extrados, the helix unwinds. With a tendon in the intrados it overwinds.}\label{fig-helical}
\end{center}
\end{figure}

All ingredients of~\eqref{hamflowinteract} are now at hand: $Q = 0$ and $\widetilde{\c t}  = \c t = - K{\widehat u}'$ from equation~\eqref{Qttendon} and the $\mathcal G$ matrix from~\eqref{Gtendon}; substituting yields
\begin{equation} \label{tendonstatics}
  \begin{cases} 
    &R' = R (\widehat {\c u} + K^{-1}\c m)^\times,\\
 & \c m' + \c u\times \c m =  l \times \c u - \frac {( (\c m +  l)\times \c u+ K{\widehat {\c u}}')\cdot (\c d_1\times Kl)} {K_{22}\,K_{33} \mathcal L_1+l\cdot Kl}\, (\c d_1 \times l).
  \end{cases}
\end{equation}

We have already established in~\eqref{hamflowso3} that on $SO(3)$ the evolution equation for the momenta is always
\[p'= p\times \c u - D_q\mathcal H,\] 
which is simpler than the second group of equations in~\eqref{tendonstatics}. However, expressing $\c u$ as a function of $p_i$ would require solving a non-linear system of equations at every integration step in $s$.

\begin{proposition}[Uniqueness of the unloaded equilibrium]\label{prop-unique}
Consider the tendon-actuated rod with one free tip and no external loads.
Then $p$ vanishes identically along the rod, so the equilibrium strain
satisfies, at every $s$,
\begin{align*}
  \c m(\c u(s)) + l(\c u(s)) = 0,
\end{align*}
and on the half-space $D = \{\lambda_2 = 1 - a\,\c u^2 > 0\}$ this
equation has exactly one solution.
\end{proposition}
 
\begin{proof}
The Hamiltonian does not depend explicitly on the configuration
($D_q\mathcal H = 0$), so by the Poisson bracket~\eqref{defbra} with the
$SO(3)$ structure constants of \textbf{Step 2} in
section~\ref{sec-singlerod}, $p\cdot p$ is constant along the flow; the
free-tip condition $p = 0$ then gives $\forall s: p(s)= 0$.
 
For the uniqueness, note that $D$ is convex and that
$\mathcal L = W + \mathcal L_1$ is smooth and strictly convex on $D$: $W$
is strictly convex because $K$ is positive-definite, and
$\mathcal L_1 = T\Lambda$ is convex because $\Lambda$ is the Euclidean
norm of an affine function of $\c u$. An interior critical point of a
convex function is a global minimizer, and strict convexity forbids two
minimizers~\cite[\textsection~27]{ro70}; hence $p = \nabla_{\!\c u}\mathcal L$
vanishes at no more than one point of $D$. Existence follows from the
explicit construction below.
\end{proof}
 
\begin{rem}
In the variables $(\lambda_2,\lambda_3)$ the equation $p = 0$ places the
solution at the intersection, within the quadrant delimited by the
reference values, of an ellipse centred at
$(\widehat\lambda_2,\widehat\lambda_3)$ with semi-axes $a^2T/K_{22}$ and
$a^2T/K_{33}$, and of a tension-independent hyperbola passing through the
reference and the origin. The equilibrium therefore travels a fixed
conic arc from the reference strain towards the origin as the tension
grows; we exploit this construction elsewhere. Physically, no unstable
or bifurcated branch exists to be missed: the strict convexity of the
energy density performs, pointwise, the role that a second-variation
analysis would otherwise play.
\end{rem}
\begin{rem} 
  When numerically solving equations~\eqref{tendonstatics} without external stresses, we can use the idea that $p\cdot p$ is constant to transfer the boundary condition from one end of the rod to the other. Of course for the unloaded robot ($T\neq 0$ but no other loads), we do not need to solve the differential equation at all, we can instead solve the nonlinear system $p(\c u) = 0$ and recover the shape in post-processing.
\end{rem}

\begin{rem} 
Consider the thought experiment of applying a tension $T$ but hiding the system in a sleeve. An operator acting on the rod ---without knowledge of its inner structure--- would experience a rod for which the transmitted moment is $p$. This rod would however display a nonlinear relation between $p$ and~${\c u}$. 
\end{rem}

\begin{figure}[t]
\begin{center}
\includegraphics[width=0.49\textwidth]{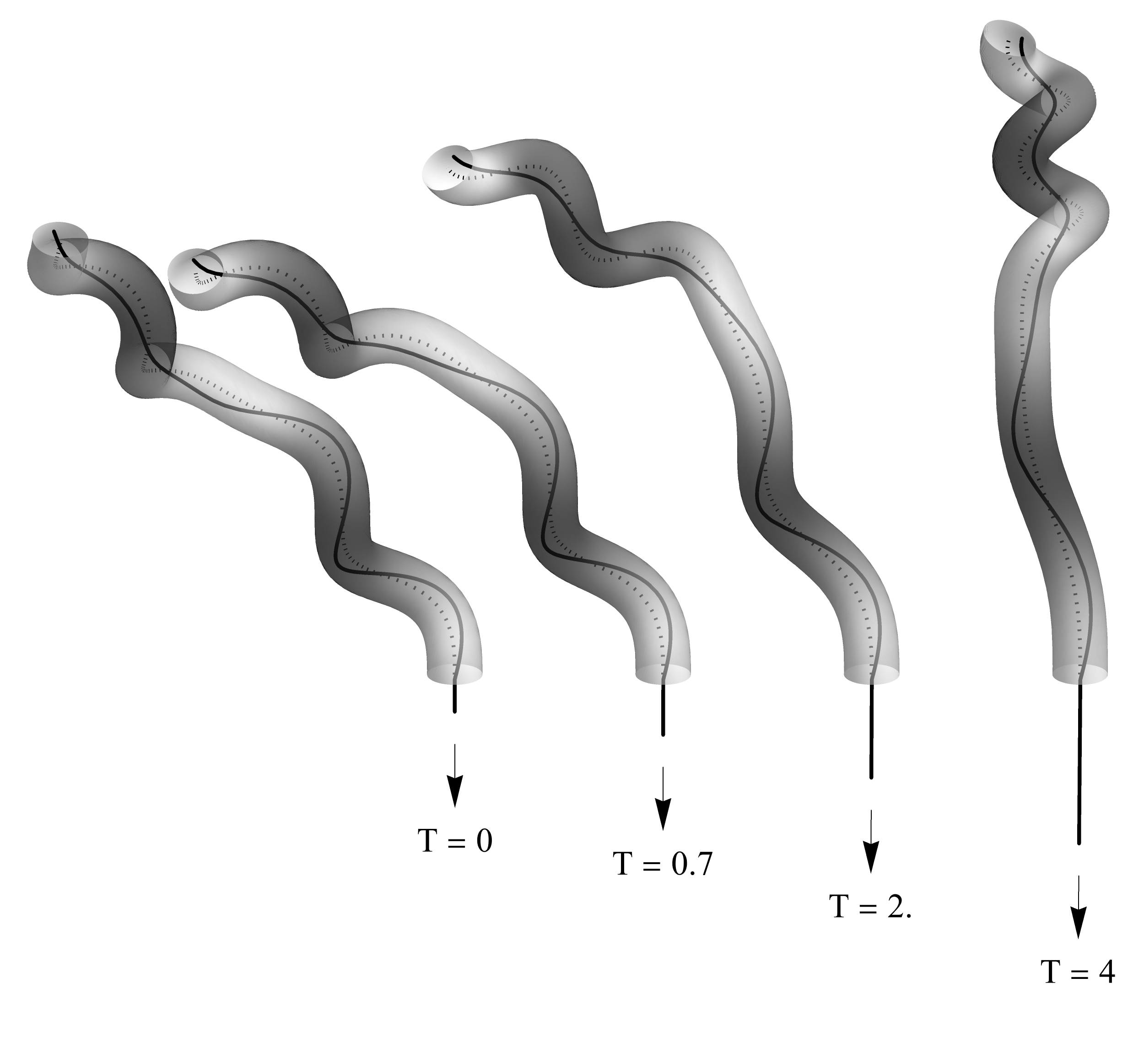}
\caption{Equilibrium configurations of a rod with a helical perversion ($\widehat{\c u}^1=0, \widehat{\c u}^2(s)=\tanh(s-L/2),$ and $\widehat{\c u}^3=6\pi/L$and actuated by a tendon ($K_{11}=K_{22}=1,\, K_{33}=0.8,\, a=0.8,\, L = 4\pi.$). The equilibria were computed by integrating~\eqref{tendonstatics} where the term with the arc length derivative of the reference configuration is now active. In keeping with the previous figure, one side of the  perversion overwinds and the other underwinds.}\label{fig-perversion}
\end{center}
\end{figure}

\section{Conclusions}
 
Casting the statics of a rod on the framed extended tangent bundle
$\overline{TQ}$ has produced the following chain. The conjugate momenta
are not chosen but forced: they are the unique functions rendering the
governing two-form nonsingular with holonomic characteristics
(Proposition~\ref{prop-pforced}), and the functional they define remains
the elastic energy on holonomic curves. For an isolated rod, the forced
momenta coincide with the material force and moment, and the Kirchhoff
equations emerge identically through three encodings of the rotation
group (section~\ref{sec-singlerod}). When an interaction couples to the
strains, the momenta decompose along the additivity of the energy, and
expressing the two factors of the Poisson bivector in different frames
lets the internal stresses carry the computation while the conjugate
momenta carry the structure (section~\ref{sec-interact}). For the
tendon-actuated rod this closes the equilibrium equations in finite
form, the required inversion collapsing to a rank-one correction, and
yields a uniqueness theorem: the unloaded robot has exactly one shape
(Proposition~\ref{prop-unique}).
 
On the way, we brought to bear two ideas often kept in separate corners
of mathematical physics: working with frames on $\overline{TQ}$, and
working with Lepage-equivalent forms. The second was developed in its
literature into a much more general tool than what is needed for rods;
Proposition~\ref{prop-pforced} gathers what is needed here and is
accompanied by a tailored proof.
 
The construction also clarifies what the two sets of variables are. The
internal stresses $(\c n,\c m)$ are subsystem quantities, defined by the
elastic energy of the rod alone; the momenta $p$ are system quantities,
selected by the variational structure of the whole problem and equal to
the total force and moment transmitted across a section. An operator
holding a tendon-actuated rod hidden in a sleeve experiences a single
effective rod whose constitutive relation between $p$ and the strain is
nonlinear; the formalism computes that effective law, and then treats
the composite as a rod in its own right. We expect this coarse-graining
to be the natural language for compound filaments.
 
Two companion directions are developed separately. Rigid constraints and
their limits are best treated on the extended cotangent bundle, where
the singular limits of the stiffnesses become regular. Stability
requires the second variation of the energy; the present construction
was arranged so that the functional \emph{is} the energy on holonomic
curves, precisely so that equilibria can later be sought as minimizers
rather than mere critical curves. Finally, the tendon example invites a
systematic treatment of actuated structures --- multiple rods, multiple
tendons, stiffnesses and offsets varying along the length, friction
between tendon and channel --- for which the equations of
section~\ref{sec-interact} provide the conservative backbone.

\appendix

\section{Variation of the integral of a form under a flow \label{app-varform}}

Let $\mathcal M$ be an $n$-dimensional manifold, $W$ a smooth vector field on $\mathcal M$, $\phi$ the flow associated with $W$, $\mathcal S\subset \mathcal M$ be a $k$-dimensional submanifold with $k<n$, and $\omega$ a $k$-form on $\mathcal M$. 

Then we let $W$ carry $\mathcal S$ along its flow. That is for all $\varepsilon$ sufficiently small  we define 
$$\mathcal S_\varepsilon = \{ \phi_\varepsilon(s) : s\in\mathcal S\},$$
where $\phi_0$ is the identity operator.

The variation of the integral 
\begin{align}
I = \int_{\mathcal S} \omega,\label{defI}
\end{align}
with respect to the vector field $W$ is then defined as 
\begin{align}
\delta I[W] := \left.\left( \sd{} \varepsilon \int_{\mathcal S_\varepsilon} \omega\right) \right |_{\varepsilon= 0}.
\end{align}

Define $\mathcal V = \bigcup_{\sigma\in[0,\varepsilon]} \mathcal S_\sigma.$ If $\mathcal S$ has a boundary $\partial \mathcal S$, then the boundary of $\mathcal V$ is $\partial \mathcal V=\mathcal S_\varepsilon \cup \mathcal B \cup \mathcal S_0$ where $\mathcal B:= \bigcup_{\sigma\in[0,\varepsilon]} \partial \mathcal S_\sigma.$ Then by Stokes and Fubini, we have the following three results : 
\begin{align}
\int_{\mathcal V} d\omega &= \int_{\mathcal S_\varepsilon -\mathcal S_0 -\mathcal B}\omega,&
 \int_{\mathcal V} d\omega 
&= \int_0^\varepsilon \left(\int_{\mathcal S_\sigma}\iota_W( d\omega)\right)\, d\sigma,&
\int_{\mathcal B} \omega & = \int_0^\varepsilon \left ( \int_{\partial S_\sigma} \iota_W(\omega)\right) \, d\sigma.
\end{align}

Reorganising the first equality, then substituting the other two therein yields
\begin{align*}
\int_{\mathcal S_\varepsilon} \omega& = \int_{\mathcal V} d\omega + \int_{\mathcal S_0} \omega + \int_{\mathcal B} \omega,\\
&= \int_0^\varepsilon \left ( \int_{\mathcal S_\sigma} \iota_W(d\omega) \right ) d\sigma + \int_{\mathcal S_0} \omega + \int_0^\varepsilon \left ( \int_{\partial \mathcal S_\sigma} \iota_W(\omega)\right )d\sigma,
\end{align*}
so that 
\begin{align}
\delta I [W] &= \int_\mathcal S \iota_W(d\omega) + \int_{\partial S} \iota_W(\omega) \label{delI2terms}\\
&= \int_\mathcal S \iota_W(d\omega)+d( \iota_W(\omega) ) \\ 
& = \int_\mathcal S L_W (\omega),\label{delILie}
\end{align}
where $L_W(\omega)$ is the Lie derivative of $\omega$ with respect to the vector field $W$.

Both expressions~\eqref{delI2terms} and~\eqref{delILie} of the first variation of $I$ are useful in practice. The former separate bulk and boundary terms while the latter gives a particularly compact expression where we can use many technical rules of calculus with Lie derivatives.

Naturally, $\delta I[W]$ is itself an integral of a $k$-form so we can consider 
\begin{align}
\delta (\delta I[W]) [U] &= \int_{\mathcal S} L_U(L_W(\omega))\\
&=:\delta^2I[W,U],
\end{align}
where the last equality defines the notation $\delta^2$.

\section{Nonsingular two-forms in odd dimensional spaces and vortex lines\label{app-vortex}}
This is a summary of the beginning of section 44 of \cite{ar00}. Let $M$ be an odd-dimensional manifold and $\omega^2$, a real-valued two-form defined on $M$. 
\begin{lem}\label{lem-nullvectexist}
For all $p\in Q$, there is a non-vanishing vector $\xi\in T_pM$ such that $$\omega^2(\xi,\eta) = 0, \qquad \forall \eta \in T_p M.$$
\end{lem}
A proof is easy to build upon defining a basis of $T_p M$ and showing that the determinant of the matrix associated with $\omega^2$ must vanish; see~\cite{ar00}.

A vector $\xi$ as defined in Lemma~\ref{lem-nullvectexist} is called a \emph{null vector} of $\omega^2$. Clearly the set of null vectors of $\omega^2$ form a vector subspace of $T_p M$. 
\begin{defi}The form $\omega^2$ is \emph{nonsingular} iff its subspace of nullvectors has dimension 1 for all $p\in Q$.
\end{defi}

Let $\omega^1$ be a 1-form on $M$ such that $d\omega^1$ is nonsingular. Next we construct a non-vanishing vector field $X$ of nullvectors of $d\omega^1$.
For every $p\in Q$, let $X(p) \in T_pM$ be a non-zero nullvector of $d\omega^1$. $X(p)$ is called the vortex direction of $\omega^1$ at $p$. The integral curves of a field $X$ of vortex directions are called the vortex lines (or characteristic lines) of the 1-form $\omega^1$.

\subsection*{Vortex lines are critical curves for variational problems} 
Let $M$ be an odd-dimensional manifold, $\omega$ a 1-form on $M$ such that $d\omega$ is non-singular,  and $p_1,\, p_2\in Q$. Consider the problem of finding a curve $\gamma$ joining $\mathcal S_1\subset M$ to $\mathcal S_2\subset M$ such that 
\be{varOform}
{\mathcal E}= \int_\gamma \omega
\ee be minimal.

A vector field $W$ is called an admissible perturbation $\gamma$ if it is continuous in an open neighbourhood of $\gamma$ and if it is tangent to the boundary conditions at both ends of the curve.

We first ask that $\gamma$ be critical, that is for all admissible vector fields $W$ we have
\begin{align*}
0&=\delta \mathcal E [W]  \stackrel{\eqref{delI2terms}}{=}  \int_\gamma \iota_W(d\omega) + \int_{\partial \gamma} \iota_W(\omega) \\
& = \int_a^b d\omega(W,\gamma')|_{\gamma(s)}\,  ds + \omega(W)|_{\gamma(b)} - \omega(W)|_{\gamma(a)}. 
\end{align*}

Because the last expression must vanish for all admissible $W$, we deduce that $\gamma'$ must be along a vortex line of $d\omega$ and that $\gamma(b)$ and $\gamma(a)$ are points of $\mathcal S$ where $\omega$ is normal to the boundary condition -- we say that a covector is \emph{normal} to a surface if it vanishes on all vectors that are tangent to the surface.

\section{Framing the tangent bundle\label{app-framing}}

The aim of this appendix is to start with a framing or a coframing of a manifold $M$ and deduce a particular framing and coframing of its extended tangent bundle.

We first need to remind the definitions of submersion of a manifolds, vertical subspace (with respect to a submersion) and smooth framings as these concepts will be central to the following discussion. A full account can be found in various textbooks in differential geometry; for instance~\cite{Lee12,st12}

Given two differentiable manifolds $A$ and $B$ a submersion $\pi:A\to B$ is a smooth mapping the pushforward of which is surjective, that is $\forall a\in A: ~~d\pi|_a:T_aA\to T_{\pi(a)}B$ is surjective.

Given such a submersion, the subspace of vertical vectors at $a\in A$ is noted $\textrm{Vert}_\pi a$ and defined as the vector subspace to $T_aA$ that is pushed forward to $0$ by $d\pi|_a$. That is 
$$\textrm{Vert}_\pi a = \{X\in T_a A: d\pi( X) = 0\in T_{\pi(a)}B\}.$$
Notably, the projection alone does not define a horizontal space (see~\cite{st12}). But we will not be needing horizontal spaces here. In particular we will not be defining connections on our vector bundles.

\subsection{Framing a manifold}
A (smooth) framing (respectively coframing) of a manifold $Q$ of dimension $n$ is a $n$-tuple of (smooth) vector fields (respectively 1-forms) on $Q$ that form a basis of $T_qQ$ (respectively $T^\star_qQ$) at every point $q\in Q$. Note that many manifolds do not admit global framings; the 2-sphere is a notorious example. Note also that all Lie groups can be globally framed because any basis of the tangent space at the identity can be extended as left-invariant vector fields over the whole group.

A familiar coframing of an $n$ dimensional open subset $U\subset Q$ is given by a set of coordinates which are $n$ functions $x^i:U\to\mathbb R^n$. Indeed their differentials $dx^i$ are expected to be linearly independent 1-forms everywhere in $U$. The dual frame is then simply formed by the partial derivatives $\pd{}{x^i}$. Note that frames defined from coordinates systems are always so that the constituting vector fields commute since partial derivatives of smooth functions commute:
\begin{align}
\forall\varphi \in\FQ: [\partial_i,\partial_j] \varphi := \pddd{\varphi}{x^i}{x^j} -  \pddd{\varphi}{x^j}{x^i}  = 0,
\end{align}
where $\FQ$ denotes the set of smooth functions on $Q$.

This is not so for general frames for which the commutators of elements is again a vector fields on $Q$. That commutator can therefore be expressed as linear combination of the frame's elements $(E_1,\cdots, E_n)$:
\begin{align*}
\forall\varphi \in\FQ: [E_i,\,E_j] \varphi := E_i(E_j(\varphi)) - E_j(E_i(\varphi)) := c^k_{ij} E_k(\varphi),
\end{align*}
where the $c^k_{ij}$ are functions on $M$ that do not depend on $\varphi$. They are the structure functions of the frame. If we note $(\theta^1,\cdots, \theta^n)$ the dual frame, we have by definition
\begin{align}
\label{defstructure}
c^k_{ij} = \theta^k([E_i,E_j]).
\end{align}
It follows (see p.311\,\cite{Lee12}) that,
\begin{align}
d\theta^i= -\frac{c^i_{jk}}{2} \,\theta^j\wedge\theta^k =-c^i_{jk}\,\theta^j\otimes \theta^k.
\end{align}

For Lie groups framed by left-invariant vector fields, the structure functions $c^k_{ij}$ are constants over the whole space.

It is worth also noting that whenever one is equipped with a smooth frame (resp. coframe) of a finite dimensional manifold, one can always obtain a smooth coframe (resp. frame) by considering the dual frame at every point. 

\subsection{Framing the (extended) tangent bundle}

Given a local coframing $\beta^1,\cdots, \beta^n$ of $Q$ defined $\forall q \in Q$, our aim is to build a frame and a coframe for $\overline{TQ}$. First, define the dual frame $B_1,\cdots, B_n$ to that of the $\beta s$. That is $B_j$ is a vector field on $Q$ such that everywhere on $Q$, we have $\forall i : \beta^i(B_j) = \delta^i_j.$

 The space $\overline{TQ}$ forms a vector bundle
 of rank $n$ over $\overline Q$ with the associated projection $\pi:\overline{TQ}\mapsto\overline Q:(q,X,s)\mapsto (q,s)$. A fibre $\pi^{-1}(q,s)$ of $\overline{TQ}$ over $(q,s)$ is isomorphic to $T_q Q$ the tangent space to $Q$ at $q$. Next we define $n$ functions over $\overline {TQ}$: 
 $$\xi^i:\overline{TQ} \to \mathbb R: (q,X,s) \mapsto \beta^i|_q(X).$$
 These functions are continuous on $\overline{TQ}$ because the $\beta^i$ are continuous 1-forms. The restrictions of each of these $n$ functions to the fibre $\pi^{-1}(q,s)$ over any point of $\overline{Q}$ provide a system of coordinates over the fibre since for any $X\in T_q Q$, we have $X = \beta^i(X) B_i = \xi^i(X) \, B_i.$ 
 
Next, let $\theta^i=\pi^\star \,\beta^i$ be the pull-backs of the $\beta$ by the projection $\pi$. Then $$\theta^1,\cdots,\theta^n,d{\xi^1},\cdots,d{\xi^n},\,ds$$ form a coframing of $\overline{TQ}$. That is, these $2n+1$ continuous 1-forms provide a basis for the cotangent space to every point of $\overline{TQ}$. We can then define the dual basis of tangent vectors $E_{q^1},\cdots, E_{q^n},E_{\xi^1},\cdots,E_{\xi^n},E_s$. They provide $2n+1$ vector fields over $\overline{TQ}$ that constitute a framing.

Also note that with the further projection to the first factor $\pi_s : \overline{Q}= Q\times [a,b] \to Q:(q,s)\mapsto Q$, we have by construction, $\pi_{s\star} \circ \pi_\star E_{q^i} = B_i$.

To compute $d\theta^i$, we use (see p.310 in~\cite{Lee12})
\be{appdtheta}
d\theta^i (X,W) =X(\theta^i(W)) - W(\theta^i(X)) - \theta^i([X,W])
\ee
To compute $\theta^i([X,W])$, we consider first that $\pi_\star E_{q^i} = B_i\in \mathfrak X(\bar M)$,  $\pi_\star E_{\xi^j} = 0 \in\mathfrak X(\bar M)$ (see the p. 10 of \cite{kono63}) and then that 
\ben{appEcomm}
\theta^i ([E_{\xi^j},E_{\xi^k}])& =& (\pi^\star \beta^i) ([E_{\xi^j},E_{\xi^k}]) = \pi^\star \big ( \beta^i  (\pi_\star [ E_{\xi^j},E_{\xi^k} ]) \big) = \pi^\star (\beta^i([\pi_\star E_{\xi^j},\pi_\star E_{\xi^k}])) = 0,\nonumber \\
\theta^i ([E_{\xi^j},E_{q^k}]) &= & \pi^\star (\beta^i([\pi_\star E_{\xi^j},\pi_\star E_{q^k}])) = 0,\nonumber \\ 
\theta^i ([E_{q^j},E_{q^k}]) &= & \pi^\star (\beta^i([\pi_\star E_{q^j},\pi_\star E_{q^k}])) =\pi^\star (\beta^i([B_j,B_k]) ) = c^i_{jk},\nonumber \\ 
\theta^i ([E_{q^j},E_{s}]) &= & \pi^\star (\beta^i([B_j,B_s])) =  \pi^\star (\pi_s^\star  \beta^i) ([B_j,B_s])) =   \pi^\star \pi_s^\star  \beta^i ([ B_j,0])=0. \nonumber 
\end{eqnarray}
where $c^i_{jk}$ are the structure functions of the basis $\beta^i$.
Finally, we compute
\ben{appDcomm}
\theta^i ([D_p^j,D_p^k])& =& \theta^i \Big(\left [(P^{-1})^{ja} E_{\xi^a}, (P^{-1})^{kb} E_{\xi^b}\right]\Big) =0,\nonumber \\
\theta^i ([D_{p}^j,D_{q^k}]) &= &  \theta^i \Big(\left [(P^{-1})^{ja} E_{\xi^a}, E_{qk} - (CP^{-1})^{b}_k E_{\xi^b}\right]\Big) =0,\nonumber\\
\theta^i ([D_{q^j},D_{q^k}]) &= &  \theta^i \Big(\left [E_{q^j} - (CP^{-1})^{a}_j E_{\xi^a}, E_{qk} - (CP^{-1})^{b}_k E_{\xi^b}\right]\Big) =c^i_{jk},\nonumber\\
\theta^i ([D_{q^j},D_{s}]) &=&0. \nonumber 
\end{eqnarray}

Gathering, we have 
\begin{eqnarray}
\theta^k([X,W]) &=& \theta^k\left(\left [X,
W_q^a \, D_{qa} + W_{pb}\, D_{p}^b + W_s \, D_{s}\right]\right) \nonumber \\
&=&\theta^k\left(  W_q^a \left [X,D_{qa} \right] + X(W_{q}^a) D_{qa} + W_{pb} \left [X,D_{p}^j\right] + X(W_{pb})\, D_p^b + W_s \left[X,D_s\right] + X(W_s)\, D_s\right)\nonumber \\
&=& W_q^a \, \theta^k\left ( \left [ X_q^i \, D_{q^i} + X_{pj}\, D_{p}^j + X_s \, D_{s},D_{qa}\right]\right) +X(W_q^k) \nonumber \\
&&\quad+ W_{pb} \theta^k\left (\left [ X_q^i \, D_{q^i} + X_{pj}\, D_{p}^j + X_s \, D_{s},D_{pb}\right]\right) + W_s \theta^k\left (\left [ X_q^i \, D_{q^i} + X_{pj}\, D_{p}^j + X_s \, D_{s},D_{s}\right]\right)\nonumber \\
&=&W_q^a X_q^i c^k_{ia} -W_q^a D_{qa} (X_q^k) + X(W_q^k) - W_{pb} D_{pb}(X_q^k) -W_s D_s(X_q^k),\nonumber \\
&=&W_q^a X_q^i c^k_{ia} + X(W_q^k) - W(X_q^k). \label{thetabracket}
\end{eqnarray}

Substituting\re{thetabracket} in\re{appdtheta}, we find  (in agreement with p.311\,\cite{Lee12})
\be{appdthetafin}
d\theta^i(X,W) = - c^i_{jk} X_q^j\, W_q^k.
\ee


\begin{thebibliography}{10}

\bibitem{dilima95}
D.~J. Dichmann, Y.~Li, and J.~H. Maddocks.
\newblock {Hamiltonian formulations and symmetries in rod mechanics}.
\newblock {\em IMA Vol. Math. Appl.}, 82:71--114, 1996.

\bibitem{kema97b}
S.~Kehrbaum and JH~Maddocks.
\newblock Elastic rods, rigid bodies, quaternions and the last quadrature.
\newblock {\em Philos. Trans. Roy. Soc. A}, 355(1732):2117--2136, 1997.

\bibitem{va01}
G.~H.~M. Van~der Heijden.
\newblock The static deformation of a twisted elastic rod constrained to lie on a cylinder.
\newblock {\em Proc. Roy. Soc. Lond. A Mat.}, 457(2007):695--715, 2001.

\bibitem{sthe14}
E.L. Starostin and G.H.M. {van der Heijden}.
\newblock Theory of equilibria of elastic 2-braids with interstrand interaction.
\newblock {\em Journal of the Mechanics and Physics of Solids}, 64:83--132, 2014.

\bibitem{eu44}
L.~Euler.
\newblock {\em {Methodus inveniendi lineas curvas maximi minimive proprietate gaudentes sive solutio problematis isoperimetrici latissimo sensu accepti}}.
\newblock apud Marcum-Michaelem Bousquet \& socios, 1744.

\bibitem{ki59}
G.~Kirchhoff.
\newblock \"{U}ber das gleichgewicht und die bewegung eines unendlich d\"{u}nnen elastischen stabes.
\newblock {\em J. Reine Angew. Math.}, 56:285--313, 1859.

\bibitem{Cosserat1909}
E.~Cosserat and F.~Cosserat.
\newblock {\em Th{\'{e}}ories des Corps D{\'{e}}formables}.
\newblock Hermann, Paris, 1909.

\bibitem{an05}
Stuart~S. Antman.
\newblock {\em {Nonlinear problems of elasticity}}, volume 107 of {\em Applied Mathematical Sciences}.
\newblock Springer New York, New York, 2 edition, 2005.

\bibitem{neva02}
S.~Neukirch and G.~H.~M. van~der Heijden.
\newblock Geometry and mechanics of uniform n-plies: from engineering ropes to biological filaments.
\newblock {\em J. Elasticity}, 69:41--72, 2002.

\bibitem{moma05}
Maher Moakher and John~H Maddocks.
\newblock A double-strand elastic rod theory.
\newblock {\em Arch. Ration. Mech. Anal.}, 177(1):53--91, 2005.

\bibitem{lemogo14}
Th. Lessinnes, D.~E. Moulton, and A.~Goriely.
\newblock Morphoelastic rods. {Part II}: {G}rowing birods.
\newblock {\em J. Mech. Phys. Solids}, 100:147--196, 2017.

\bibitem{moulton2020morphoelastic}
Derek~E Moulton, Thomas Lessinnes, and Alain Goriely.
\newblock {Morphoelastic rods III: Differential growth and curvature generation in elastic filaments}.
\newblock {\em J. Mech. Phys. Solids}, 142:104022, 2020.

\bibitem{vachth02}
G.~H.~M. Van~der Heijden, A.~R. Champneys, and J.~M.~T. Thompson.
\newblock Spatially complex localisation in twisted elastic rods constrained to a cylinder.
\newblock {\em Int. J. Solids and Struct.}, 39(7):1863--1883, 2002.

\bibitem{ch03}
N.~Chouaieb.
\newblock {\em Kirchhoff's problem of helical solutions of uniform rods and their stability properties}.
\newblock PhD thesis, {\'E}cole Polytechnique F{\'E}d{\'E}rale de Lausanne, 2003.

\bibitem{poincare01}
Henri Poincar{\'e}.
\newblock Sur une forme nouvelle des {\'e}quations de la {M}{\'e}canique.
\newblock {\em Comptes rendus hebdomadaires des s{\'e}ances de l'Acad{\'e}mie des sciences}, 132:369--371, 1901.
\newblock S{\'e}ance du 18 f{\'e}vrier 1901.

\bibitem{ki62}
A.~A. Kirillov.
\newblock Unitary representations of nilpotent {L}ie groups.
\newblock {\em Russian Mathematical Surveys}, 17(4):53--104, 1962.

\bibitem{ar66}
V.~I. Arnold.
\newblock Sur la g\'eom\'etrie diff\'erentielle des groupes de {L}ie de dimension infinie et ses applications \`a l'hydrodynamique des fluides parfaits.
\newblock {\em Annales de l'Institut Fourier}, 16(1):319--361, 1966.

\bibitem{ko70b}
Bertram Kostant.
\newblock Quantization and unitary representations.
\newblock In {\em Lectures in Modern Analysis and Applications {III}}, volume 170 of {\em Lecture Notes in Mathematics}, pages 87--208. Springer, Berlin, 1970.

\bibitem{so70}
J.-M. Souriau.
\newblock {\em Structure des syst\`emes dynamiques}.
\newblock Dunod, Paris, 1970.
\newblock English translation: {\emph{Structure of Dynamical Systems: A Symplectic View of Physics}}, Birkh\"auser, 1997.

\bibitem{ki04}
A.~A. Kirillov.
\newblock {\em Lectures on the Orbit Method}, volume~64 of {\em Graduate Studies in Mathematics}.
\newblock American Mathematical Society, 2004.

\bibitem{maje99}
Jerrold~E. Marsden and Tudor~S. Ratiu.
\newblock {\em Introduction to Mechanics and Symmetry}, volume~17 of {\em Texts in Applied Mathematics}.
\newblock Springer, New York, 2 edition, 1999.

\bibitem{cema01}
Hern\'an Cendra, Jerrold~E. Marsden, and Tudor~S. Ratiu.
\newblock Lagrangian reduction by stages.
\newblock {\em Memoirs of the American Mathematical Society}, 152(722):x+108, 2001.

\bibitem{cema03}
Hern{\'a}n Cendra, Jerrold~E. Marsden, Sergey Pekarsky, and Tudor~S. Ratiu.
\newblock Variational principles for {L}ie--{P}oisson and {H}amilton--{P}oincar{\'e} equations.
\newblock {\em Moscow Mathematical Journal}, 3(3):833--867, 2003.

\bibitem{homa98}
Darryl~D. Holm, Jerrold~E. Marsden, and Tudor~S. Ratiu.
\newblock The {E}uler--{P}oincar\'e equations and semidirect products with applications to continuum theories.
\newblock volume 137, pages 1--81. 1998.

\bibitem{blma09}
Anthony~M. Bloch, Jerrold~E. Marsden, and Dmitry~V. Zenkov.
\newblock Quasivelocities and symmetries in nonholonomic systems.
\newblock {\em Dyn. Syst.}, 24(2):187--222, 2009.

\bibitem{ca22}
E.~Cartan.
\newblock {\em Le\c{c}ons sur les invariants int{\'e}graux}.
\newblock Hermann, Paris, 1922.

\bibitem{le36}
T.~Lepage.
\newblock Sur les champs g\'eod\'esiques des int\'egrales multiples.
\newblock {\em Bull. Acad. Roy. Belg., Cl. Sci.}, 27:27--46, 1936.

\bibitem{kr73}
D.~Krupka.
\newblock Some geometric aspects of variational problems in fibered manifolds.
\newblock {\em Folia Fac. Sci. Nat. Univ. Purkyn. Brun., Physica}, 14:14, 1973.
\newblock Also available at arXiv:math-ph/0110005.

\bibitem{kr15}
D.~Krupka.
\newblock {\em Introduction to Global Variational Geometry}.
\newblock Atlantis Press, Paris, 2015.

\bibitem{goriely17}
Alain Goriely.
\newblock {\em {The mathematics and mechanics of biological growth}}, volume~45 of {\em Interdisciplinary applied mathematics}.
\newblock Springer-Verlag, New York, 1 edition, 2017.

\bibitem{pila13}
A.~Pirrera, X.~Lachenal, S.~Daynes, P.~M. Weaver, and I.~V. Chenchiah.
\newblock Multi-stable cylindrical lattices.
\newblock {\em J. Mech. Phys. Solids}, (61):2087--2107, 2013.

\bibitem{ar00}
V.~I. Arnold.
\newblock {\em Mathematical Methods of Classical Mechanics}.
\newblock Springer, 2000.

\bibitem{mara99}
Jerrold~E. Marsden and Tudor~S. Ratiu.
\newblock {\em Introduction to Mechanics and Symmetry}.
\newblock Springer, 2nd edition, 1999.

\bibitem{gri83}
Phillip~A. Griffiths.
\newblock {\em Exterior Differential Systems and the Calculus of Variations}.
\newblock Birkhauser, 1983.

\bibitem{Lee12}
John~M. Lee.
\newblock {\em Introduction to Smooth Manifolds}, volume 218 of {\em Graduate Texts in Mathematics}.
\newblock Springer, New York, 2nd edition, 2012.

\bibitem{hamel04}
Georg Hamel.
\newblock Die {L}agrange-{E}ulerschen {G}leichungen der {M}echanik.
\newblock {\em Zeitschrift f{\"u}r Mathematik und Physik}, 50:1--57, 1904.

\bibitem{lima87}
Paulette Libermann and Charles-Michel Marle.
\newblock {\em Symplectic Geometry and Analytical Mechanics}, volume~35 of {\em Mathematics and its Applications}.
\newblock D. Reidel Publishing Company, Dordrecht, 1987.

\bibitem{al86}
Simon~L. Altmann.
\newblock {\em Rotations, Quaternions, and Double Groups}.
\newblock Oxford University Press, Oxford, 1986.
\newblock Reprinted by Dover Publications, Mineola, NY, 2005.

\bibitem{bole23}
Fr{\'e}d{\'e}ric Boyer, Vincent Lebastard, Fabien Candelier, Federico Renda, and Mazen Alamir.
\newblock Statics and dynamics of continuum robots based on cosserat rods and optimal control theories.
\newblock {\em IEEE transactions on robotics}, 39(2):1544--1562, 2023.

\bibitem{ruwe14}
D~Caleb Rucker and Robert~J Webster~III.
\newblock Mechanics of continuum robots with external loading and general tendon routing.
\newblock In {\em Experimental Robotics: The 12th International Symposium on Experimental Robotics}, pages 645--654. Springer, 2014.

\bibitem{ro70}
R.~Tyrrell Rockafellar.
\newblock {\em Convex Analysis}.
\newblock Princeton University Press, Princeton, NJ, 1970.

\bibitem{st12}
Shlomo Sternberg.
\newblock {\em Curvature in Mathematics and Physics}.
\newblock Dover Publications, Mineola, NY, 2012.

\bibitem{kono63}
Shoshichi Kobayashi and Katsumi Nomizu.
\newblock {\em Foundations of Differential Geometry}, volume~I of {\em Interscience Tracts in Pure and Applied Mathematics}.
\newblock Interscience Publishers, John Wiley \& Sons, New York, 1963.
\newblock Reprinted in Wiley Classics Library, 1996.

\end{thebibliography}
\end{document}